\documentclass[11pt]{article}

\usepackage{inputenc}
\usepackage{amsmath}
\usepackage{amsfonts}
\usepackage{amssymb}
\usepackage{amsthm}
\usepackage{graphicx}
\usepackage[T1]{fontenc}
\usepackage[english]{babel}
\usepackage{mathrsfs}
\usepackage{amsthm}
\usepackage[all]{xy}
\usepackage{geometry}
\usepackage{enumerate}
\usepackage{pgfplots}
\usepackage{subfigure}
\pgfplotsset{compat=1.15}
\allowdisplaybreaks
\usepackage{xcolor}
\usepackage{appendix}
\usepackage{authblk}
\usepackage[colorlinks=true]{hyperref} 
\usepackage{cleveref}
\usepackage{dsfont}
\usepackage{mathtools}

\newtheorem{theorem}{Theorem}

\newtheorem{lemma}[theorem]{Lemma}

\newtheorem{proposition}[theorem]{Proposition}

\newtheorem{definition-proposition}[theorem]{Definition-Proposition}

\newtheorem{remark}[theorem]{Remark}

\newcommand{\N} {\mathbb{N}}

\newcommand{\C} {\mathbb{C}}
\newcommand{\R} {\mathbb{R}}

\DeclarePairedDelimiter\abs{\lvert}{\rvert}%
\DeclarePairedDelimiter\norm{\lVert}{\rVert}%
\newcommand{\tnorm}[1]{{\left\vert\kern-0.25ex\left\vert\kern-0.25ex\left\vert #1 
    \right\vert\kern-0.25ex\right\vert\kern-0.25ex\right\vert}}

\makeatletter
\let\oldabs\abs
\def\abs{\@ifstar{\oldabs}{\oldabs*}}
\let\oldnorm\norm
\def\norm{\@ifstar{\oldnorm}{\oldnorm*}}
\makeatother

\newcommand*\diff{\mathop{}\!\mathrm{d}}

\newcommand{\Tr}{\operatorname{Tr}}

\newcommand{\sinc}{\operatorname{sinc}}

\title{\sf Minimal time of magnetization switching in small ferromagnetic ellipsoidal samples}
\author[1]{Rapha\"el C\^ote\footnote{\texttt{\small raphael.cote@unistra.fr}} }
\author[1]{Cl\'ementine Court\`es \footnote{ \texttt{\small clementine.courtes@unistra.fr}} }
\author[1]{Guillaume Ferri\`ere \footnote{ \texttt{\small guillaume.ferriere@unistra.fr}} }
\author[1,2]{Yannick Privat \footnote{\texttt{\small  yannick.privat@unistra.fr}} }

\affil[1]{IRMA, Universit\'e de Strasbourg, CNRS UMR 7501, Inria, 7 rue Ren\'e Descartes, 67084 Strasbourg, France. }
\affil[2]{Institut Universitaire de France (IUF)}

\begin{document}
\maketitle

\begin{abstract}
In this paper, we consider a ferromagnetic material of ellipsoidal shape. 
The associated magnetic moment then has two asymptotically stable opposite equilibria, of the form $\pm\overline{m}$. In order to use these materials for memory storage purposes, it is necessary to know how to control the magnetic moment. We use as a control variable a spatially uniform external magnetic field and consider the question of flipping the magnetic moment, i.e., changing it from the $+\overline{m}$ configuration to the $-\overline{m}$ one, in minimal time. Of course, it is necessary to impose restrictions on the external magnetic field used. We therefore include a constraint on the $L^\infty$ norm of the controls, assumed to be less than a threshold value $U $. 
We show that, generically with respect to the dimensions of the ellipsoid, there is a minimal value of $U $ for this problem to have a solution. We then characterize it precisely. Finally, we investigate some particular configurations associated to geometries enjoying symmetries properties and show that in this case the magnetic moment can be controlled in minimal time without imposing a threshold condition on $U $. 
\end{abstract}


\medskip



\section{Introduction}
\subsection{The Landau-Lifshitz equation for ellipsoidal ferromagnetic samples}
Ferromagnetic materials have come into common use in the last few decades, especially since they are found in devices used to store digital information such as magnetic tapes or hard disks, but also in magnetic chips called \emph{Magnetic Random Access Memory} (MRAM). These chips have many advantages over their silicon counterparts, in particular that of requiring energy only to change the value bits and not to maintain the storage itself. This is probably one of the most challenging applications since it opens the door towards new spintronic applications and storage technologies while allowing a very fast access to information (see, e.g., \cite{parkin2008magnetic}). 

The magnetic moment of a ferromagnetic material represented by a domain $\Omega\subset\R^3$ is usually modelled as a time-varying vector field
\[
m:\R\times\Omega\rightarrow\mathbb{S}^2,
\]
where $\mathbb{S}^2$ is the unit sphere of $\R^3$, the evolution of which is driven by the so-called \emph{Landau-Lifshitz equation} (see \cite{landau2013electrodynamics})
\begin{equation}\label{eqLL}
 \frac{\partial m}{\partial t}=-m\wedge h(m)-\alpha m\wedge (m\wedge h(m)),
\end{equation}
where the \emph{effective field} $h(m)$ is defined by
\[
h(m) = 2A\Delta m+h_d(m)+h_{\mathrm{ext}}
\]
with $\alpha>0$, a constant (in time and space) damping coefficient which is characterized by the material. We refer for instance to \cite{hubert2008magnetic,brown1963micromagnetics} for additional explanations.
The constant $A>0$ is the exchange constant, and can be assumed to be equal to $A=1/2$ without loss of generality, with a normalization argument. The \textit{demagnetizing field} $h_d(m)$ is the solution of the equations 
\[
\left\{
\begin{aligned}
&\mathrm{div}(h_d(m)+m)=0\\
&\mathrm{curl}(h_d(m))=0
\end{aligned}\right.\quad \text{ in }\mathcal{D}'(\R^3)
\]
where $m$ is extended to $\R^3$ by $0$ outside $\Omega$ and  $\mathcal{D}'(\R^3)$ denotes the space of distributions on $\R^3$. The field $h_{\mathrm{ext}}$ is an external one, for instance it can be an external magnetic field.

Note that it is possible to complete and specify this physical model by adding other relevant terms, for example by taking into account the anisotropic behavior of the crystal that composes the ferromagnetic material.

Finally, it is standard to assume homogeneous Neumann boundary conditions on the magnetization  on the boundary of $\Omega$.

\medskip

In this article, we will consider a ferromagnetic sample of ellipsoidal shape, and the magnetization $m$ and external field $h_{\rm ext}$ both spatially uniform. Indeed, ellipsoidal domains have been much studied in the literature dedicated to ferromagnetism \cite{osborn1945demagnetizing,di2016newtonian,takahashi2017ellipsoids}: on the one hand, they cover a large variety of geometrical shapes, and on the other hand,  they are the only known bodies that can be uniformly magnetized in the presence of a spatially uniform inducing field. From the mathematical point of view, it is nice to consider such samples because the demagnetizing field $h_d$ appearing in the Landau-Lifschitz equation can be determined in an explicit way. 

Let us be more precise and clarify the model obtained in this case. In all the following of this article, we will denote $\Omega$ the ellipsoid of $\R^3$ of semiaxes $a_1>0$, $a_2>0$ and $a_3>0$, and a basis $(O;e_1,e_2,e_3)$ chosen so that the Cartesian equation of $\Omega$ reads
\begin{equation}\label{eq_ellips}
\frac{x^2}{a_1^2}+\frac{y^2}{a_2^2}+\frac{z^2}{a_3^2}=1.
\end{equation}
An illustration of the ellipsoid $\Omega$ is shown in Figure \ref{fig:ellipsoide}.
According to \cite{osborn1945demagnetizing,di2016newtonian}, for uniform (in space) magnetizations $m$ on $\Omega$, the demagnetizing field $h_d(m)$ can be explicitly computed and reads 
\[
h_d(m)=-Dm, \qquad\text{with}\quad D=\begin{pmatrix}
\gamma_1 & 0 & 0\\
0 & \gamma_2 & 0\\
0 & 0 & \gamma_3
\end{pmatrix} ,
\]
where $\gamma_i$ ($i=1,2,3$) denotes a positive constant depending only on the semiaxes $a$, $b$ and $c$ (we provide the precise dependence in Appendix~\ref{append:demag_ellips}).

One can easily infer from this result that, provided that the external field $h_{\rm ext}$ and the initial magnetization $m^0$ are constant in space,  so is the magnetic moment $m$ solving the Landau-Lifshitz equation \eqref{eqLL} completed with homogeneous Neumann boundary conditions.

As a consequence, the Landau-Lifschitz equation with a time-dependent external magnetic field $h_{\rm ext}$ reads as the ordinary differential system
\begin{equation}\label{LL:ODE}
\left\{
\begin{aligned}
& \dot{m}  = \alpha \left(h_0(m)-(h_0(m)\cdot m) m\right)-m\wedge h_0(m) \quad \text{in }(0,T) \\
& m(0) =m^0
\end{aligned}\right.
\end{equation}
where the dotted notation $\dot{m}$ stands for the time derivative of $m$, $h_0(m)=-Dm+h_{\rm ext}$, $T>0$, $m(t)\in\mathbb{S}^2\subset\R^3$, $D=\operatorname{diag}([\gamma_1,\gamma_2,\gamma_3])$ denotes a diagonal matrix with positive coefficients.
Up to a change of basis, we will also assume without loss of generality that
\begin{equation}\label{hyp:alpha}
0\leq \gamma_1\leq \gamma_2\leq\gamma_3\leq 1.
\end{equation}

In what follows, we will assume that the ferromagnetic particle is subjected to a spatially uniform external magnetic field $h_{\rm ext}$, and we are interested in two asymptotically stable stationary states of the resulting system, denoted $\overline{m}$. We seek to answer the following question: 
\begin{quote}
\textit{Given a maximum value $U $ of the norm of the field $h_{\rm ext}$ at all times, can we determine whether there exists such a field flipping the magnetic spin from $\overline{m}$ to $-\overline{m}$ in minimal time?}
\end{quote}

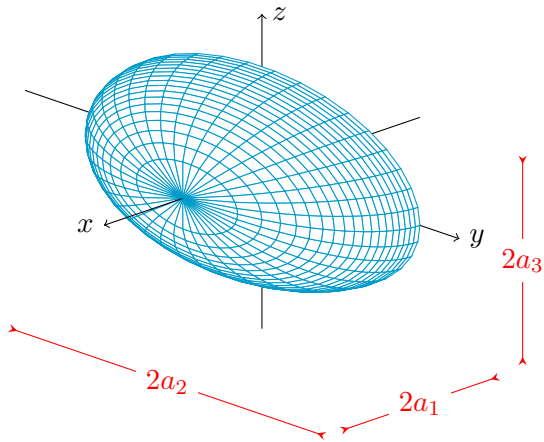
\begin{figure}[htbp]
\begin{center}
\begin{tikzpicture}[scale=2]
\pgfsetlayers{pre main,main,axis foregound}
\begin{axis}
[view={135}{20},colormap={blue}{
            color=(cyan) color=(cyan)
        },axis lines=none,axis equal,set layers=standard,
enlargelimits,domain=0:2,samples=20, z buffer=sort,
]

\pgfonlayer{axis background}
\draw[-] (axis cs:0,0,-1.5)--(axis cs:0,0,-1);
\draw[-] (axis cs:0,-3,0)--(axis cs:0,-2,0);
\draw[-] (axis cs:-2,0,0)--(axis cs:-1,0,0);
\addplot3[draw=none] coordinates{(0,0,-2) (0,0,2)};
\addplot3[draw=none] coordinates{(0,-2.5,0) (0,2.8,0)};
\endpgfonlayer
\pgfonlayer{main}
\addplot3 [surf,draw=cyan,fill=white,domain=0:1,samples=20,
domain y=00:180] ({x},{-2*cos(y)*sqrt(1-x^2)},{-sin(y)*sqrt(1-x^2)});
\addplot3 [surf,draw=cyan,fill=white,domain=0:1,
domain y=0:180,on layer=axis foreground] ({x},{2*cos(y)*sqrt(1-x^2)},{sin(y)*sqrt(1-x^2)});
\coordinate(x1) at (axis cs:1,0,0);
\coordinate(x2) at (axis cs:2,0,0);
\coordinate(y1) at (axis cs:0,1.9,0);
\coordinate(y2) at (axis cs:0,2.5,0);
\coordinate(z1) at (axis cs:0,0,0.9);
\coordinate(z2) at (axis cs:0,0,1.5);
\coordinate(x1bis) at (axis cs:-1,2,-1.7);
\coordinate(x2bis) at (axis cs:1,2,-1.7);
\coordinate(y1bis) at (axis cs:1.2,-2,-1.7);
\coordinate(y2bis) at (axis cs:1.2,2,-1.7);
\coordinate(z1bis) at (axis cs:0,3.3,-1);
\coordinate(z2bis) at (axis cs:0,3.3,1);
\endpgfonlayer
\end{axis}
\draw[->] (x1)--(x2)node[left]{$x$};
\draw[->] (y1)--(y2)node[right]{$y$};
\draw[->] (z1)--(z2)node[right]{$z$};
\draw [>-< , 
    > = stealth, red ] (x1bis) -- (x2bis) node [midway, fill = white] {$2a_1$} ;
\draw [>-< , 
    > = stealth, red ] (y1bis) -- (y2bis) node [midway, fill = white] {$2a_2$} ;
    \draw [>-< , 
    > = stealth, red ] (z1bis) -- (z2bis) node [midway, fill = white] {$2a_3$} ;
    \end{tikzpicture}
\end{center}
\caption{The ellipsoid shaped ferromagnetic sample \label{fig:ellipsoide}}
\end{figure}

\subsection{State of the art and structure of the article}

The development of the use of ferromagnetic materials has led to the emergence of new storage possibilities, and consequently to a renewed interest of the scientific community around the control of EDO/EDP on this topic.

The use of an external magnetic field to control a ferromagnetic system is a very present issue in the literature of the field (see for instance \cite{MR2375581,MR3348399,MR1773088,MR3011326}).

Many works have focused on both the derivation of relevant models, i.e. sufficiently close to the physics, but also simple enough to be exploited mathematically, and on the related optimization issues. Many studies are devoted to these modeling questions, to the obtaining of exploitable optimality conditions leading to numerical simulations. Thus, in the same spirit as the present study, the authors of \cite{MR4419351} seek to flip the magnetic spin using electric current injections. Let us mention in the same vein the works \cite{MR3407264,MR3961301} also addressing similar issues: minimization of the distance to a target state with a fixed time horizon, addition of stochastic term in the model, search for a feedback and numerical analysis of the considered problems. 

Recent progress has been made in the understanding of the control (exact and approximate) of ellipsoidal samples/networks : \cite{dubey2019controllability,5399599,agarwal2011control}. Our study has been particularly motivated by \cite{alouges2009magnetization}, in which it is notably proved that, when the size of an open bounded by $\Omega$ tends to 0, then we find a uniform magnification in the domain, which lends itself to the study of ellipsoids. 

\paragraph{Structure of the article. }In this paper, we are interested in a single ferromagnetic particle of ellipsoidal shape in $\R^3$. We seek to perform a magnetic moment reversal in minimal time, using an external magnetic field as a control of the resulting physical system. 

We model this issue in section~\ref{sec:OCP}, imposing a maximum $L^\infty$ norm on the control translated using the parameter $U>0$, reflecting the difficulty and cost of using very high magnetic fields. In the absence of additional symmetries on the geometry of the system, we show the existence of a minimal threshold on $U$ for the minimum time problem to have a solution. We refine this result when additional symmetries are assumed on the material geometry. The main results of this paper are gathered in the section~\ref{sec:mainRes}. The section~\ref{sec:minOptPrelim} contains the foundation of the proofs of the main results: indeed, we state necessary and sufficient conditions guaranteeing the well-posedness of the time-optimal problem and write the necessary conditions of optimality to the first order using the Pontryagin maximum principle. 

The proofs of the main results are contained in the sections~\ref{proof:theo2}, \ref{sec:ProofSym} and \ref{proof:theo8}. Finally, some numerical simulations are listed in the section \ref{sect:Conclusion and perspectives} to illustrate the qualitative behavior of the solutions obtained in theoretical theorems. The appendices contain additional information and/or secondary calculations. Appendix \ref{append:demag_ellips} contains the calculation of the demagnetizing field in the case of a ferromagnetic ellipsoid sample. Appendix \ref{append:e1AS} contains the proof that $-e_1$ is indeed the only asymptotically stable state for equation \eqref{LL:ODE}. Finally, Appendix \ref{append:comp} contains the calculations of the explicit constants in the case $\gamma_1<\gamma_2$.

\paragraph{Notations.}
In the whole article, $|\cdot|$ denotes the standard euclidean norm on $\R^3$ (or $\R^d$), and its inner product in $\R^d$ is denoted with a dot.

We are essentially interested in a control problem where the control function is the external field: we abide by the usual convention, and denote 
\[ h_{\rm ext}=u. \]

\section{Existence of a minimal switching time}
Let us recall that, as mentioned in the introduction, we will consider a  ferromagnetic sample whose shape $\Omega$ is the ellipsoid with Cartesian equation \eqref{eq_ellips}. The dynamics of the magnetic moment $m(\cdot)$, equal to $m_0$ at the initial time, is hence driven by the simplified Landau-Lifshitz equation \eqref{LL:ODE}.

\subsection{Towards an optimal control problem}\label{sec:OCP}

The main issue we want to tackle reads
\begin{quote}
{\it Given a steady-state $\overline{m}$ of \eqref{LL:ODE} in $\mathbb{S}^2$, can we achieve a reversal by solving an optimal control problem, i.e. steering the system from $m(0)=\overline{m}$ to $m(T)=-\overline{m}$ while minimizing $T$?}
\end{quote}
In what follows, we will consider particular stationary states: $\overline{m}=\pm e_1$. It is proved in appendix~\ref{append:e1AS} that these equilibria are asymptotically stable when $\gamma_1 < \gamma_2$. Therefore, they can be used as magnetic spin orientation for memory storage purposes. 
We will denote by $u(\cdot)$ the external (spatially uniform) magnetic field imposed on the system. This is the control variable in this problem. The question is then to ask if it is possible to steer the solution $m_u$ of the system \eqref{LL:ODE} associated to the control $u(\cdot)$ and to the initial data $m_u(0)=e_1$ until $m_u(T)=-e_1$, in minimal time. 

Of course, it is necessary to add physical constraints to this problem: if one imposes no restrictions on the choice of admissible controls, it is likely that the minimal time problem will have a solution, reached by unrealistic controls.
For this reason, we will assume in what follows the constraint
\begin{equation}\label{constr:barU}\tag{$\mathscr{C}_{T,U }$}
|u(t)|\leq U \quad \text{a.e. in }(0,T), 
\end{equation}
in order to limit the choice of controls to realistic possibilities.

All in all, the problem we aim at investigating reads as follows.

\begin{quote}
\textbf{Minimal time problem: }{\it let $U >0$ and assume that $m_0=e_1$. The problem reads
\begin{equation}\label{OCP:minT}\tag{$\mathscr{P}_0$}
\mathcal{T}_{U } \coloneqq \inf_{(T,u)\in \mathcal{O}_{U }}T,
\end{equation}
where 
$$
 \mathcal{O}_{U }=\{(T,u)\mid T\in \R_+, \ u\in L^\infty(0,T)\text{ satisfies }\eqref{constr:barU}\text{ and }m_u(T)=-e_1\},
$$
with $m_u$, the solution to \eqref{LL:ODE} associated to the control function $u(\cdot)$ and the initial datum $e_1$.}
\end{quote} 

We will investigate the following issues:
\begin{itemize}
\item Does Problem~\eqref{OCP:minT} have an optimal solution for any value of $U >0$ ?
\item How to characterize all the solutions to this problem and understand their geometric dependence to the parameters $\gamma_i$, $i=1,2,3$?
\end{itemize}

\subsection{Main results}\label{sec:mainRes}

First, the minimization problem is indeed well-posed, meaning that the existence of an optimal solution is equivalent to the existence of a minimal trajectory.

\begin{theorem} \label{th:ex_min}
    Let $U  > 0$. The following properties are equivalent:
    \begin{itemize}
        \item[(i)] There exists an optimal pair $(\mathcal{T}_{U}, u) \in \mathcal{O}_{U }$ for Problem~\eqref{OCP:minT}.
        \item[(ii)] $\mathcal{T}_{U}$ is finite.
        \item[(iii)] $\mathcal{O}_{U }$ is nonempty.
    \end{itemize}
\end{theorem}

The behavior of the control system differs greatly depending on the values of the parameters $\gamma_i$ and more specifically on the values of $\gamma_1$ and $\gamma_2$. 

\begin{theorem} \label{th:ex_U_crit}
    Assume $\gamma_1 < \gamma_2$. Then there exists $U_\textnormal{crit} > 0$ such that
    \begin{itemize}
        \item for all $U \in (0, U_\textnormal{crit}]$, \eqref{OCP:minT} has no solution.
        \item for all $U > U_\textnormal{crit}$, \eqref{OCP:minT} has a solution.
    \end{itemize}
\end{theorem}

\begin{remark}
It is notable that the proof of this theorem provides an explicit lower-bound estimate of $U_\textnormal{crit}$. The precise bound is derived in Remark \ref{rk:theoremUcrit}
\end{remark}

We are now interested in the case where $\gamma_1=\gamma_2$, which is not covered by the above result. It is interesting to note that in this case, the behavior of the optimized physical system is very different from the one described in the Theorem~\ref{th:ex_U_crit}. Indeed, this situation of symmetry leads to the fact that there is no longer a threshold from which the system is controllable.

To complete this analysis, we also investigate in the following result the existence of optimal planar trajectories. In view of the system symmetry, it is natural to conjecture that the optimal trajectory are planar, since all the points in $\operatorname{span} (e_1, e_2) \cap \mathbb{S}^2$ are stable, and this set is even asymptotically stable. Somewhat surprisingly, We show here that this is actually not the case.

\begin{theorem}\label{theo:gamma1=gamma2}
If $\gamma_1=\gamma_2\leq \gamma_3$, then the optimal control problem~\eqref{OCP:minT} has a solution whatever the value of $U >0$, meaning that $U_{\textrm{crit}}=0$, with the notations of Theorem~\ref{th:ex_U_crit}. Furthermore, if $\gamma_1<\gamma_3$, the optimal trajectory in $\mathbb{S}^2$ is not contained in the plane $\operatorname{span}(e_1,e_2)$. 
\end{theorem}

It is interesting to notice that Theorem~\ref{th:ex_U_crit} can be refined in the particular case where $\gamma_1 < \gamma_2 = \gamma_3$. To this aim, we will deeply exploit the necessary first order optimality conditions provided by the so-called Pontryagin Maximum Principle (PMP). We refer to Section~\ref{optCond:PMP} for a precise statement of such conditions.

\begin{theorem}\label{theo:gamma2=gamma3}
    If $\gamma_1\leq \gamma_2 = \gamma_3$, then $U_\textnormal{crit} = \frac{\alpha}{2 \sqrt{1 + \alpha^2}} (\gamma_2 - \gamma_1)$. Furthermore, for all $U > U_\textnormal{crit}$, 
    $$
    \mathcal{T}_U = \frac{\pi}{\sqrt{1 + \alpha^2} \sqrt{U^2 - U_\textnormal{crit}^2}}.
    $$
\end{theorem}

\begin{remark}
    In particular, we infer from the result above the following asymptotics:
    \begin{equation*}
        \mathcal{T}_U \sim \frac{\pi}{\sqrt{2 U_\textnormal{crit} (1 + \alpha^2)}} \frac{1}{\sqrt{U - U_\textnormal{crit}}}\quad \text{as }U \searrow U_\textnormal{crit},
    \end{equation*}
   and
    \begin{equation*}
        \mathcal{T}_U \sim \frac{\pi}{U \sqrt{1 + \alpha^2}}\quad \text{as }U\to +\infty.
    \end{equation*}
\end{remark}
\begin{remark}[Case of where the shape of the sample is a sphere]
In the case where $\Omega$ is a sphere, then one has $\gamma_1=\gamma_2=\gamma_3$. Then, both conclusions of Theorem~\ref{theo:gamma1=gamma2} and \ref{theo:gamma2=gamma3} apply, meaning that $U_{\textrm{crit}}=0$ and the optimal time is given by
$\mathcal{T}_U  =\pi/(U\sqrt{1 + \alpha^2})$. Furthermore, it may be shown that, in that case, there exists optimal planar trajectories in each of the hyperplanes $\operatorname{span}(e_i,e_j)$ with $i\neq j$. We refer for instance to \cite[Proof of Prop.~2]{alouges2009magnetization}, whose main argument can be reproduced in our case.
\end{remark}

We end this section by a result on the asymptotic behavior of optimal magnetization trajectories as $U$ diverges to $+\infty$. We prove that optimal trajectories tend to be supported on a geodesic on the sphere whenever $U$ is large.

\begin{theorem} \label{th:alm_planar}
    Let $\gamma_1\leq \gamma_2\leq \gamma_3$ and $U > U_\textnormal{crit}$ and $m$ be an optimal trajectory. Let $p$ be its adjoint state. Then, if $U$ is large enough, $m$ stays close to the plane $V = \operatorname{span} (e_1, p(\mathcal{T}_U))$ in the following sense: there exists $U_0 > 0$ and $C > 0$ such that for every $U > U_0$ and $t \in [0, \mathcal{T}_U]$,
    \begin{equation*}
        \norm{m (t) - \mathbb{P}_V m (t)} \leq \frac{C}{U},
    \end{equation*}
    where $\mathbb{P}_V$ denotes the orthogonal projection onto $V$.
\end{theorem}

\section{Minimization and optimality}\label{sec:minOptPrelim}

\subsection{Proof of Theorem~\ref{th:ex_min}: existence of an optimal trajectory}

    Let us assume that $\mathcal{O}_{U }$ is nonempty. This allows us to consider a minimizing sequence $(T_n,u_n)_{n\in\N}\in \mathcal{O}_{U }^\N$, and $m_n \in \mathscr{C}([0,T_n]$ the solution to \eqref{LL:ODE} with field $u_n$.
    By definition, $T_n \rightarrow \mathcal{T}_U$ as $n \rightarrow \infty$.
    In what follows, we will denote similarly a sequence and any subsequence with a slight abuse of notation, for the sake of simplicity. 
    
    Let us introduce the functions $\widetilde{u}_n$, $\widetilde{m}_{u_n}$ defined on $[0,1]$ by
    $$
    \widetilde{u}_n(s)=u_n(T_ns)\quad \text{and}\quad \widetilde{m}_{n}(s)=m_{n}(T_ns).
    $$
    Hence, System~\eqref{LL:ODE} rewrites
    \begin{equation}\label{LL:ODEbis}
    \left\{\begin{array}{ll}
     \dot{\widetilde{m}}_n = T_n \left(\alpha \left(\widetilde{h}_0(\widetilde{m}_n)-(\widetilde{h}_0(\widetilde{m}_n)\cdot \widetilde{m}_n) \widetilde{m}_n\right)-\widetilde{m}_n\wedge \widetilde{h}_0(\widetilde{m}_n)\right)& \text{in }(0,1) \\
     \widetilde{m}_n(0) =e_1 &
    \end{array}\right.
    \end{equation}
    where $\widetilde{h}_0(\widetilde{m}_n)=-D\widetilde{m}_{n}+\widetilde{u}_n$.
    
    Similarly, since the sequence $(\widetilde u_n)_{n\in\N}$ is bounded in $L^\infty(0,1)$, it converges weakly-star in $L^\infty(0,1)$ up to a subsequence to some element $u^*$ such that $|u^*(\cdot)|\leq U $ a.e. in $[0,1]$ according to the Banach-Alaoglu-Bourbaki theorem. Since both $(\widetilde m_{n})_{n\in \N}$ and $(\widetilde u_n)_{n\in \N}$ are bounded in $L^\infty([0,1])$, we infer that so is $\dot{\widetilde m}_{n}$ according to \eqref{LL:ODE}. Therefore, the sequence $(\widetilde m_{n})_{n\in \N}$ is bounded in $W^{1,\infty}(0,1)$ and hence converges (up to a subsequence) towards an element $\widetilde{m}^*\in W^{1,\infty}(0,1)$ in $C^0([0,1])$ according to the Ascoli theorem. In particular, one has necessarily $|\widetilde{m}^*(\cdot)|=1$. Now, let us rewrite \eqref{LL:ODEbis} as the fixed-point equation
    $$
    \forall s\in [0,1], \quad \widetilde m_{n}(s)=e_1+T_n\int_0^s\left( \alpha \left(\widetilde h_0(\widetilde m_{n})-(\widetilde h_0(\widetilde m_{n})\cdot \widetilde m_{n}) \widetilde m_{n}\right)-\widetilde m_{n}\wedge \widetilde h_0(\widetilde m_{n})\right)\, d\sigma
    $$
    Observe that the right-hand side is linear with respect to $\widetilde h_0(m_{n})$ and that, according to the properties above, $(\widetilde h_0(\widetilde m_{n}))_{n\in \N}$ converges weakly-star to $\widetilde h_0(\widetilde m^*)$ in $L^\infty(0,1)$, where $\widetilde{h}_0(\widetilde{m}^*)=-D\widetilde{m}^*+\widetilde{u}^*$. Letting $n$ tend to $+\infty$ in the equation above, we obtain:
    \[ 
    \forall s\in [0,1], \quad \widetilde m^*(s)=e_1+\mathcal{T}_U\int_0^s\left( \alpha \left(\widetilde h_0(\widetilde m^*)-(\widetilde h_0(\widetilde m^*)\cdot \widetilde m^*) \widetilde m^*\right)-\widetilde m^*\wedge \widetilde h_0(\widetilde m^*)\right)\, d\sigma.
    \]
    Moreover, since $\tilde{m}_n (1) = - e_1$ by construction, the convergence in $\mathcal{C}^0 ([0, 1])$ leads to $\tilde{m}^* (1) = - e_1$. Taking the previous formula with $s=1$, we get
    \begin{equation*}
        - 2 e_1 = \mathcal{T}_U\int_0^1\left( \alpha \left(\widetilde h_0(\widetilde m^*)-(\widetilde h_0(\widetilde m^*)\cdot \widetilde m^*) \widetilde m^*\right)-\widetilde m^*\wedge \widetilde h_0(\widetilde m^*)\right)\, d\sigma.
    \end{equation*}
    This proves that $\mathcal{T}_U > 0$.
    Now, let us introduce $u^*$ as $u^*(t)=\widetilde{u}^*(t/\mathcal{T}_U)$. By undoing the change of variable above, we get that $\widetilde{m}^*(t/\mathcal{T}_U)=m_{u^*}(t)$ for a.e. $t\in [0,\mathcal{T}_U]$. 
    Furthermore, $m_{u^*}(0)=e_1$ and $m_{u^*}(\mathcal{T}_U)=-e_1$ since $\widetilde{m}_n(0)=e_1$ and $\widetilde{m}_n(1)=-e_1$. The converse sense is straightforward and the expected conclusion follows. 

Finally, observe that the same reasoning can be reproduced whenever $\mathcal{T}_U$ is finite.

\subsection{Sufficient and necessary condition for the existence of an admissible trajectory}

Let $(m,u)$ be a solution to \eqref{LL:ODE} on $[0,T]$, and for $t \in [0,T]$, consider the mobile frame $\mathcal{B}(t)=(m(t),\dot e(t),m(t)\wedge \dot e(t))$ where $\dot{e}=\dot{m}/|\dot{m}|$\footnote{Here, $\dot e$ is merely a notation, and not the time derivative of a previously defined vector $e$.}. According to \cite{alouges2009magnetization}, by observing that 
\[
m \perp \dot m, \quad \dot m \perp m \wedge \dot m \quad \text{and}\quad m \perp m \wedge (m\wedge Dm),
\]
one shows easily, by decomposing $u(t)$ into $\mathcal{B}(t)$ and writing the equation for $u - (u \cdot m) m$, the projection of $u$ on $m^\perp$, that there exists $\lambda\in L^\infty (0,T)$ such that 
\begin{equation} \label{eq:u_m}
u = \frac{1}{1+\alpha^2}(\alpha \dot m + m \wedge \dot m)+Dm-(Dm\cdot m)m+\lambda m.
\end{equation}
In fact, $\lambda = u\cdot m$. Reciprocally, given any function $m \in W^{1,\infty}([0,T],\mathbb S^2)$, and any function $\lambda \in L^\infty([0,T],\mathbb R)$, if we define $u$ by \eqref{eq:u_m}, then $(m,u)$ is solution to \eqref{LL:ODE}. These considerations can be seen as a consequence of a flatness property of the main system. 

Again assuming that $(m,u)$ is admissible trajectory, i.e  a solution to \eqref{LL:ODE}, We infer from \eqref{eq:u_m} that $u(t)$ expands as
\begin{equation}\label{eq:uexpand}
u=\lambda m +\left(\frac{\alpha}{1+\alpha^2}|\dot m|+\dot e\cdot Dm \right)\dot e+\left(Dm \cdot (m \wedge \dot e)+\frac{1}{1+\alpha^2}|\dot m|\right)m \wedge \dot e.
\end{equation}
As, a consequence, using that $Dm \cdot (m \wedge \dot e)=\dot e\cdot (Dm \wedge m )$ due  the triple product property, we get
\begin{align*}
|u|^2 &= \lambda^2+\left(\frac{\alpha}{1+\alpha^2}|\dot m|+\dot e\cdot Dm\right)^2+\left(Dm\cdot (m\wedge \dot e)+\frac{|\dot m|}{1+\alpha^2}\right)^2\\
&= \lambda^2+\left(\dot e\cdot Dm+\frac{\alpha|\dot m|}{1+\alpha^2}\right)^2+\left(\dot e\cdot (Dm_u\wedge m)+\frac{|\dot m|}{1+\alpha^2}\right)^2.
\end{align*}
Clearly, this computation and the previous remarks show that, without loss of generality, we can furthermore assume that an \emph{optimal} trajectory satisfies $\lambda=0$, or equivalently, $u \cdot m =0$: we will do this in the following.

Let us introduce, for a given $T>0$, 
\[
\mathscr{V}_T=\{m\in H^1([0,T];\mathbb{S}^2)\mid m(0)=e_1\text{ and }m(T)=-e_1\}.
\]

To investigate the existence of an admissible trajectory, it is then convenient to introduce
\begin{align}\label{ob:aux}
\Lambda(T)&:=\inf_{\lambda\in L^\infty([0,T])}\inf_{m\in \mathscr{V}_T}\sup_{t\in [0,T]}\left(\lambda^2+\left(\dot e\cdot (Dm\wedge m)+\frac{|\dot m|}{1+\alpha^2}\right)^2+\left(\dot e\cdot Dm+\frac{\alpha|\dot m|}{1+\alpha^2}\right)^2\right),\nonumber \\
&= \inf_{m\in \mathscr{V}_T}\sup_{t\in [0,T]}\left(\left(\dot e\cdot (Dm\wedge m)+\frac{|\dot m|}{1+\alpha^2}\right)^2+\left(\dot e\cdot Dm+\frac{\alpha|\dot m|}{1+\alpha^2}\right)^2\right).
\end{align}

We summarize the above discussion in the form of a lemma.

\begin{lemma}\label{lem:2057}
The existence of an admissible trajectory for Problem~\eqref{OCP:minT} comes to the existence of $m\in \mathscr{V}_T$ such that the function $u$ given by \eqref{eq:uexpand} with $\lambda =0$ satisfies $\Vert u\Vert_{L^\infty([0,T];\mathbb{R}^3)}\leq U $, which is also equivalent to $\Lambda(T)\leq U ^2$. Also, $u$ satisfies $u \cdot m=0$.
\end{lemma}

\subsection{Necessary optimality conditions for Problem~\texorpdfstring{\eqref{OCP:minT}}{P0}}\label{optCond:PMP}
This problem can be solved by using the Pontryagin maximum principle. 
The main results of this section are gathered in Proposition~\ref{prop:1stOrderCond}, at the end of this section.

The Hamiltonian associated to Problem~\eqref{OCP:minT} is
$$
\mathcal{H}:\begin{array}[t]{rcl}
\mathbb{S}^2 \times \R^3\times \{-1,0\}\times \R^3 & \to & \R\\
 (m,p,p^0,u) &\mapsto & p\cdot \left(- \alpha m \wedge (m \wedge h_0(m)) -m\wedge h_0(m)\right).
\end{array}
$$
It can be noted that the dependence of $\mathcal{H}$ on the control function $u$ is affine and one has
$$
\mathcal{H} (m,p,p^0,u)=p\cdot \left(- \alpha m \wedge (m \wedge u) - m \wedge u \right)-p\cdot  \left(- \alpha m \wedge (m \wedge Dm)-m\wedge Dm\right)
$$

As a first remark, the magnetization stays in $\mathbb{S}^2 = \partial \mathbb{B}$, where $\mathbb{B}$ is the closed unit ball of $\mathbb{R}^3$. Therefore, our problem is obviously equivalent to the problem with restricted conditions
\begin{equation*}
\mathcal{T}_{U} = \inf_{\substack{(T,u)\in \mathcal O_{U} \\ \abs{m}^2 - 1 = 0}}T.
\end{equation*}
Thus, we use the version of the Pontryagin maximum principle with restricted phase coordinates, as stated in \cite[Theorem~22]{MR0186436}. This theorem is stated for a minimization of an integral with fixed $T$, but it can be easily adapted to the case of a minimal time with classical changes (see for instance the passage from Theorem 1 to Theorem 2 in the same reference). With such a statement, we point out that, for all $m \in \mathbb{R}^3$
\begin{equation*}
    \nabla ( \abs{m}^2 - 1 ) = 2 m,
\end{equation*}
and that
\begin{equation*}
    2 m \cdot \left(- \alpha m \wedge (m \wedge h_0(m)) -m\wedge h_0(m)\right) = 0.
\end{equation*}
Therefore, in our case, this statement gives the exact same necessary conditions, with an additional orthogonality condition for the adjoint state, stated hereafter.

The first order optimality conditions read as follows: let us denote by $(T,u)$, an optimal pair for this problem; there exists an absolutely continuous mapping $p:[0,T]\to \R^3$ called {\it adjoint state} and a real number $p^0\in \{0,-1\}$ such that the pair $(p,p^0)$ is non-trivial and for almost every $t\in [0,T]$, the following conditions hold.

\begin{itemize}
\item {\it Adjoint equations.} 
Setting
\begin{eqnarray*}
F_1(m,p)&\coloneqq&\alpha \Bigl(p \wedge (m \wedge Dm) + Dm \wedge (m \wedge p) - D (m \wedge (m \wedge p)) \Bigr)-Dm\wedge p-D(p\wedge m),\\
F_2(m,p,u) &\coloneqq & - \alpha(p \wedge (m \wedge u) + u \wedge (m \wedge p))+u\wedge p,
\end{eqnarray*}
one gets
\begin{equation} \label{eq:adj_eq}
    \dot p=-\frac{\partial \mathcal{H}}{\partial m}=F_1(m,p)+F_2(m,p,u).
\end{equation}
Remark that, since $\abs{m} = 1$, one equivalently has
\begin{eqnarray*}
F_1(m,p)&=&\alpha (Dp-(Dm\cdot m)p-2(m\cdot p)Dm - 2 (Dm \cdot p) m)-Dm\wedge p-D(p\wedge m),\\
F_2(m,p,u) &= & \alpha(p\cdot m)u+\alpha (u\cdot m)p - 2 \alpha (p \cdot u) m+u\wedge p 
\end{eqnarray*}
\item {\it Maximality conditions.} For a.e. $t\in [0,T]$, $u(t)$ solves the optimization problem
\begin{equation}\label{eqMaxInst1}
\max_{|v|\leq U } \mathcal{H} (m(t),p(t), p^0, v)
\end{equation}
and one has at the final time $T$
\begin{equation}\label{eqMaxInst2}
\max_{|v_T|\leq U }\mathcal{H} (m(T),p(T),p^0,v_T)=-p^0.
\end{equation}
\item {\it A useful identity.} Since the dynamics only depends on the magnetization $m(\cdot)$ and the control $u(\cdot)$, the Hamiltonian functional is constant in time:
\begin{equation}\label{HamiltonCst}
\mathcal{H} (m(t),p(t),p^0,u(t))=-p^0, \quad t\in [0,T],
\end{equation}
according to \eqref{eqMaxInst2}, by evaluating the expression for $t=T$.
\item {\it An orthogonality condition for the adjoint state.} At the final time $t=T$, the adjoint state $p (T)$ is tangent to the boundary $\abs{m}^2 - 1 = 0$ at $m (T) = - e_1$. This condition is thus equivalent to 
\begin{equation}\label{eq:orth:pT}
p (T) \cdot e_1 = 0.
\end{equation}
\item {\it Orthogonality between $u$ and $m$.}
As seen in Lemma \ref{lem:2057}, $u \cdot m=0$ on $[0,T]$.
\end{itemize}

 
\begin{remark}
Since the initial and final state are fixed, there is no need to impose any transversality condition on the adjoint state.
\end{remark}

Let us analyze the conditions \eqref{eqMaxInst1} and \eqref{eqMaxInst2}.

\paragraph{The adjoint state $p$ cannot vanish on $[0,T]$.} Indeed, in the converse case, if there exists $t_0\in [0,T]$ such that $p(t_0)=0$, it follows from the Cauchy-Lipschitz theorem that $p(\cdot)=0$ and by using Condition~\eqref{eqMaxInst2}, one gets $p^0=0$, a contradiction with the non-triviality of the pair $(p,p^0)$.

\paragraph{On condition~\eqref{eqMaxInst1}.}
Observe that $v\mapsto \mathcal{H} (m(t),p(t),p^0,v)$ is affine with respect to $v$. According to the Karush-Kuhn-Tucker theorem, there exists $\mu\geq 0$ such that $\nabla_v \mathcal{H} (m(t),p(t),p^0,u (t))-\mu u (t)=0$ and the slackness condition $\mu(|u (t)|^2-U ^2)=0$ is satisfied. 

\medskip

If the set $\mathcal I := \{|u|<U \}$ is of positive Lebesgue measure, then one has \[ \alpha (p(t)-(p(t)\cdot m(t))m(t))=p(t)\wedge m(t) \quad \text{a.e. } t \in \mathcal I. \]
Taking the scalar product of this identity with $p(t)$ leads to $|p(t)|^2=(p(t)\cdot m(t))^2$ on $\mathcal I$. Since $p$ does not vanish, it follows from the equality case in the Cauchy-Schwarz inequality that $p(t)$ is proportional to $m(t)$. We will show that such a case cannot occur.

Let us introduce the function $\varphi \coloneqq p-(p\cdot m)m$. One can compute that $\varphi$ satisfies the differential (linear) relation
\begin{multline}\label{eq1155}
\dot{\varphi}=\alpha D\varphi-\alpha (Dm\cdot m)\varphi -Dm\wedge \varphi-D(\varphi\wedge m) + u\wedge \varphi 
 +(m\cdot (Dm\wedge \varphi))m-\alpha (u\cdot \varphi)m
\end{multline}
a.e. in $[0,T]$. This follows from an easy but lengthy computation, and from the fact that the function $\lambda$ given by $\lambda=p\cdot m$ satisfies
\[
\dot{\lambda}= -2\alpha (Dm\cdot m)\lambda-(Dm\wedge p)\cdot m-\alpha (u\cdot p) + 2 \alpha (Dm \cdot p),
\] 
a.e. in $[0,T]$. We leave the details to the reader.


Now, let us assume that the set $\mathcal{I}:=\{|u(\cdot)|<U \}$ is of positive Lebesgue measure. According to the discussion above, there exists a bounded function $\lambda$ such that $p=\lambda m$ on $\mathcal{I}$, and therefore, $\varphi$ vanishes on $\mathcal{I}$.
Due to \eqref{eq1155} being linear in $\varphi$, we obtain that $\varphi (\cdot)= 0$ on $[0, T]$, which means that 
\[ p(T) = (p(T) \cdot m(T)) m(T) = (p(T) \cdot e_1) e_1 = 0, \]
from the orthogonality condition \eqref{eq:orth:pT}. But recall that  $p$ cannot vanish on $[0, T]$: we reached a contradiction.

We conclude that $|u|=U $ a.e. on $[0,T]$. It follows that 
\[ \alpha (p-(p\cdot m)m)=p\wedge m+\mu u \quad \text{on  } [0,T], \]
and using furthermore that 
\begin{align*}
|\alpha (p-(p\cdot m)m)-p\wedge m|^2&= \alpha^2(|p|^2-(p\cdot m)^2)+|p\wedge m|^2\\
&= (\alpha^2+1)(|p|^2-(p\cdot m)^2),
\end{align*} 
one gets an expression of $u$ in terms of $p$ or $\varphi$:
\begin{equation}\label{expr:uopt18091}
u = \frac{U }{(\alpha^2+1)^{1/2}} \frac{\alpha (p-(p\cdot m)m)-p\wedge m}{\sqrt{|p|^2-(p\cdot m)^2}} = \frac{U }{(\alpha^2+1)^{1/2}} \frac{\alpha \varphi - \varphi\wedge m}{\abs{\varphi}}
\end{equation}


In particular, we get
\begin{equation*}
    m \wedge u = \frac{U }{(\alpha^2+1)^{1/2}} \frac{\alpha m \wedge \varphi - \varphi}{\abs{\varphi}},
\end{equation*}
\begin{equation*}
    \alpha m \wedge (m \wedge u) = \frac{U }{(\alpha^2+1)^{1/2}} \frac{- \alpha^2 \varphi - \alpha m \wedge \varphi}{\abs{\varphi}}.
\end{equation*}
Substituting those terms in \eqref{LL:ODE} and \eqref{eq1155}, we get at last
\begin{equation} \label{eq:LL_ODE_bis}
    \dot{m} = m \wedge Dm + m \wedge (m \wedge Dm) + U  (\alpha^2+1)^{1/2} \frac{\varphi}{\abs{\varphi}},
\end{equation}
\begin{equation}\label{eq:complete_eq_phi}
\dot{\varphi}=\alpha D\varphi-\alpha (Dm\cdot m)\varphi -Dm\wedge \varphi-D(\varphi\wedge m)
 +(m\cdot (Dm\wedge \varphi))m- U  (\alpha^2+1)^{1/2} \abs{\varphi} m.
\end{equation}

\paragraph{On condition~\eqref{eqMaxInst2}.}
By setting $p(T)=(0,p_{2,T},p_{3,T})$ (since $p(T) \cdot e_1 = 0$) and $v_T=(v_{1,T},v_{2,T},v_{3,T})$, since $m(T)=-e_1$, this condition also rewrites
\[
-p^0=\max_{|v_T|\leq U }p(T)\cdot \begin{pmatrix}
0 \\ \alpha v_{2,T}-v_{3,T}\\ \alpha v_{3,T}+v_{2,T}
\end{pmatrix}=\max_{v_{2,T}^2+v_{3,T}^2= U ^2}\begin{pmatrix} v_{2,T}\\ v_{3,T}\end{pmatrix} \cdot \begin{pmatrix}
\alpha p_{2,T}+p_{3,T}\\ \alpha p_{3,T}-p_{2,T}
\end{pmatrix}. \]
The Cauchy-Schwarz inequality then implies that
$ - p^0 = U \sqrt{1 + \alpha^2} \sqrt{p_{2, T}^2 + p_{3, T}^2}$.
It follows that $p^0=-1$ (else, the pair $(p^0,p)$ would be trivial) and condition \eqref{eqMaxInst2} finally rewrites:
\begin{equation*}
    U  \sqrt{1 + \alpha^2} \abs{\varphi (T)} = - p^0 = 1.
\end{equation*}

\paragraph{Analysis of the optimality conditions.}
From the previous discussion, $u(t)$ is given by \eqref{expr:uopt18091} for a.e. $t \in [0, T]$, leading to
\begin{align*}
\abs{\varphi} \left(\alpha\left(u-(u\cdot m) m\right)-m\wedge u\right) & = A\left(\alpha^2 \varphi -\alpha \varphi\wedge m-\alpha m\wedge \varphi+m\wedge (\varphi\wedge m)\right)\\
& = A(\alpha^2 + 1) \varphi
\end{align*}
where $A=\frac{U }{(\alpha^2+1)^{1/2}}$, so that
\begin{align*}
\mathcal{H} (m(t),p(t),p^0,u(t))&=U (\alpha^2+1)^{1/2} \abs{\varphi} -p\cdot  \left(\alpha\left(Dm-(Dm\cdot m) m\right)-m\wedge Dm\right)\\
& = U (\alpha^2+1)^{1/2} \abs{\varphi} -Dm\cdot  \left(\alpha\left(p-(p\cdot m) m\right)-p\wedge m\right)\\
&= U (\alpha^2+1)^{1/2} \abs{\varphi} -Dm\cdot  \left(\alpha \varphi - \varphi \wedge m\right)\\
&= \frac{(\alpha^2+1)^{1/2} \abs{\varphi}}{U }\left(U ^2-Dm\cdot u\right)
\end{align*}
and we infer that 
\begin{equation} \label{eq:expr_max_ham}
\abs{\varphi (t)} \left(U ^2-Dm(t)\cdot u(t)\right)=\frac{U }{(\alpha^2+1)^{1/2}} > 0\quad \text{a.e. in $[0,T]$.}
\end{equation}

\paragraph{On condition \eqref{HamiltonCst}}

From the previous discussion, we have for any $t \in [0, T]$
\begin{equation}\label{eqMaxInst3}
\max_{|v|\leq U }\mathcal{H} (m(t),p(t),p^0,v)=-p^0 = 1,
\end{equation}
which leads at $t = 0$ to
\begin{equation} \label{eq:in_data_adj_cond}
    U  \sqrt{1 + \alpha^2} \abs{\varphi (0)} = - p^0 = 1.
\end{equation}
More generally, with the above expression \eqref{expr:uopt18091} of $u (t)$ which is the argmax of $\mathcal{H}$, we get
\begin{equation} \label{eq:H_const}
    1 = U  \sqrt{1 + \alpha^2} \abs{\varphi} - \varphi \cdot \Bigl[ \alpha Dm - m \wedge Dm \Bigr].
\end{equation}

\medskip

For the sake of clarity, we sum-up all these informations in the following result.

\begin{proposition}[Necessary first order optimality conditions]\label{prop:1stOrderCond}
Let $(T,u)$ denote an optimal pair for Problem~\eqref{OCP:minT}. Then, the adjoint state $p$ defined by \eqref{eq:adj_eq} does not vanish on $[0,T]$ and one has
\begin{equation}\label{expr:uopt1809}
u=\frac{U }{(\alpha^2+1)^{1/2}} \frac{\alpha \varphi  - \varphi \wedge m}{\abs{\varphi}},
\end{equation}
where $\varphi$ is given by $\varphi=p-(p\cdot m)m$.
In particular, one has $|u(t)|=U $ a.e. on $[0,T]$.

Moreover, $\varphi$ satisfies the differential relation \eqref{eq:complete_eq_phi} completed by the conditions \eqref{eq:expr_max_ham} and
\begin{equation}\label{cond:phi:timeTtime0}
 U  \sqrt{1 + \alpha^2} \abs{\varphi (0)}=U  \sqrt{1 + \alpha^2} \abs{\varphi (T)} = 1. 
\end{equation}

Finally, $m$ satisfies \eqref{eq:LL_ODE_bis}.
\end{proposition}
 
\section{Proof of Theorem~\ref{th:ex_U_crit}}\label{proof:theo2}
\subsection{Preliminary results}
We first state preliminary results, in the form of a series of lemmas.

\begin{lemma} \label{lem:cont_flow}
    For all $T < \infty$, the map
    \begin{align*}
        L^\infty (0, T) &\rightarrow W^{1, \infty} (0, T) \\
        u &\mapsto m \text{ solution to \eqref{LL:ODE} with } m(0) = e_1
    \end{align*}
    is continuous and locally Lipschitz.
\end{lemma}

\begin{proof}
    Let $u_1, u_2 \in L^\infty (0, T)$ and $m_1, m_2$ the corresponding solution to \eqref{LL:ODE}. Define $\delta m \coloneqq m_1 - m_2$ and $\delta u \coloneqq u_1 - u_2$.
Simple, though tedious, calculations provide that $\delta m$ satisfies
    \begin{multline*}
        \frac{d\delta m}{dt} = \alpha \Bigl[ - D \, \delta m + \delta u - ((- D \, \delta m + \delta u) \cdot m_1) m_1 \\
        - ((- D m_2 + u_2) \cdot \delta m) m_1 - ((- D m_2 + u_2) \cdot m_2) \delta m \Bigr] \\ - m_1 \wedge (- D \delta m + \delta u) - \delta m \wedge (- D m_2 + u_2).
    \end{multline*}
    in $(0,T)$. Since $\abs{m_1} = \abs{m_2} = 1$, we obtain
    \begin{equation} \label{eq:est_delta_m_dot}
        \abs{\frac{d\delta m}{dt}} \leq \left( (4 \alpha + 2) \norm{D}_2 + (2\alpha+1) \abs{u_2} \right) \abs{\delta m} + (2 \alpha + 1) \abs{\delta u}, \quad \text{in }(0,T),
    \end{equation}
    where $\norm{\cdot}_2$ denotes the operator norm associated to the euclidean norm $\abs{\cdot}$.
    Since $\delta m (0) = 0$, we have for all $t \in (0, T)$
    \begin{equation*}
        \abs{\delta m (t)} \leq \int_0^t \abs{\frac{d\delta m}{dt}} (s) \diff s \leq (2 \alpha + 1) T \norm{\delta u}_{L^\infty} + \int_0^t \left( (4 \alpha + 2) \norm{D}_2 + (2\alpha+1) \norm{u_2}_{L^\infty} \right) \abs{\delta m} (s) \diff s,
    \end{equation*}
    and thus by Gronwall's lemma,
    \begin{equation*}
        \abs{\delta m (t)} \leq (2 \alpha + 1) T \norm{\delta u}_{L^\infty} \exp{\left( ((4 \alpha + 2) \norm{D}_2 + (2\alpha+1) \norm{u_2}_{L^\infty}) t \right)}, \quad t\in [0,T].
    \end{equation*}
    Using this estimate, and plugging it aslo in \eqref{eq:est_delta_m_dot}, we get
    \[ \| \delta m \|_{W^{1,\infty}} \le C(T, \| u_2 \|_{L^\infty}) \| \delta u \|_{L^\infty}, \]
    and the conclusion follows.
\end{proof}


\begin{lemma} \label{lem:reach_fin_time}
    If $\gamma_1 < \gamma_2$, there exists $\delta > 0$ such that for all $U > 0$, if $\abs{m (0) + e_1} < \delta$, then $- e_1$ can be reached in finite time with a control $u$ such that $\abs{u} \leq U$.
\end{lemma}

\begin{proof}
    Let us introduce $m$ as the solution to \eqref{LL:ODE} with the feedback control term
    \begin{equation*}
        u (t) = \frac{U}{\sqrt{1 + \alpha^2}} \frac{\alpha (- e_1 + (e_1 \cdot m(t)) m(t)) + e_1 \wedge m(t)}{\sqrt{1- (m(t) \cdot e_1)^2}},
    \end{equation*}
    so that the equation on $m$ becomes autonomous, and is well defined as long as $m (t) \neq \pm e_1$. Observing that 
    $$
   \left(m,\frac{e_1-(m\cdot e_1)m}{\sqrt{1-(m\cdot e_1)^2}},\frac{m\wedge e_1}{\sqrt{1-(m\cdot e_1)^2}}\right)
    $$
    is an orthonormal basis, one immediately gets that $|u(t)|=U$ for a.e. $t\in [0,T]$.
    
Denote $m=(m_1,m_2,m_3)$ the coordinates of $m$. From \eqref{LL:ODE}, the ODEs satisfied by $m_2$ and $m_3$ are
    \begin{align*} \notag
        \dot m_2 &= - \alpha [ (\gamma_2 - \gamma_1) m_2 - ((\gamma_2 - \gamma_1) m_2^2 + (\gamma_3 - \gamma_1) m_3^3) m_2 ] + (\gamma_1 - \gamma_3) m_1 m_3 + v_2, \\
        \dot m_3 &= - \alpha [ (\gamma_3 - \gamma_1) m_3 - ((\gamma_2 - \gamma_1) m_2^2 + (\gamma_3 - \gamma_1) m_3^3) m_3 ] - (\gamma_1 - \gamma_2) m_1 m_2 + v_3.
    \end{align*}
Therefore, by setting $\widetilde{m} \coloneqq (m_2, m_3)$, it follows that $\widetilde{m}$ solves the controlled system
    \begin{equation} \label{eq:ode_tilde_m_2}
        \dot{\tilde{m}} = A_- \tilde{m} + \xi_- + \tilde v,
    \end{equation}
    where
    \begin{eqnarray*}
        A_- &=& 
            \begin{bmatrix}
                - \alpha (\gamma_2 - \gamma_1) & (\gamma_3 - \gamma_1) \\
                - (\gamma_2 - \gamma_1) & - \alpha (\gamma_3 - \gamma_1)
            \end{bmatrix}, \\
        \xi_- &=& 
            \begin{bmatrix}
                \alpha ((\gamma_2 - \gamma_1) m_2^2 + (\gamma_3 - \gamma_1) m_3^2) m_2 - (\gamma_3 - \gamma_1) (1 + m_1) m_3 \\
                \alpha ((\gamma_2 - \gamma_1) m_2^2 + (\gamma_3 - \gamma_1) m_3^2) m_3 + (\gamma_2 - \gamma_1) (1 + m_1) m_2
            \end{bmatrix}
    \end{eqnarray*}
    and $\tilde v = (v_2, v_3)$ where $v = (v_1,v_2,v_3)=\alpha (u - (u \cdot m) m) - m \wedge u$, which means here 
    \begin{equation*}
        v = U \sqrt{1 + \alpha^2} \frac{(- e_1 + m_1 (t) m(t))}{\sqrt{\abs{m(t)}^2 - (m_1 (t))^2}} = U \sqrt{1 + \alpha^2} \frac{(- e_1 + m_1 (t) m(t))}{\abs{\tilde{m}}},
    \end{equation*}
    We infer that $\tilde{v} = U \sqrt{1 + \alpha^2} m_1 \tilde{m} /\abs{\tilde{m}}$.
   Observing that $ ( 1 - m_1 ) (1 + m_1) = \abs{\tilde{m}}^2$ yields, as soon as $m_1 \leq 0$,
    \begin{equation*}
        \abs{\xi_- (t)} \leq (1 + \abs{\alpha}) \, \delta \gamma_+ \, \abs{\tilde{m} (t)}^3
    \end{equation*}
    where $\delta \gamma_+ \coloneqq \gamma_3 - \gamma_1 > 0$, and also that
    \begin{equation*}
        \abs{\tilde{v} + U \sqrt{1 + \alpha^2} \frac{\tilde{m} (t)}{\abs{\tilde{m}}}} \leq U \sqrt{1 + \alpha^2} \abs{\tilde{m} (t)}^2.
    \end{equation*}
    With these estimates and by taking the inner product of \eqref{eq:ode_tilde_m_2} with $\tilde{m}$, we get
    \begin{equation*}
        \frac{1}{2} \frac{\diff}{\diff t} \abs{\tilde m (t)}^2 \leq - U \sqrt{1 + \alpha^2} \abs{\tilde{m} (t)} + (U \sqrt{1 + \alpha^2} + \norm{A_-}) \abs{\tilde{m} (t)}^2 + (1 + \abs{\alpha}) \, \delta \gamma_+ \, \abs{\tilde{m} (t)}^4,
    \end{equation*}
    and
    \begin{equation*}
       \frac{\diff}{\diff t} \abs{\tilde m (t)} \leq - U \sqrt{1 + \alpha^2} + (U \sqrt{1 + \alpha^2} + \norm{A_-}) \abs{\tilde{m} (t)} + (1 + \abs{\alpha}) \, \delta \gamma_+ \, \abs{\tilde{m} (t)}^3.
    \end{equation*}
    Let us introduce $\delta_U \in (0,1/2)$ small enough (depending on $U > 0$) so that
    \begin{equation}\label{choice:deltaU}
        (U \sqrt{1 + \alpha^2} + \norm{A_-}) \delta_U + (1 + \abs{\alpha}) \, \delta \gamma_+ \, \delta_U^3 \leq \frac{U \sqrt{1 + \alpha^2}}{2}.
    \end{equation}
Then, if $\abs{m (0) + e_1} < \delta_U$, which gives $\abs{\tilde{m} (0)} < \delta_U$, one has 
    \begin{equation*}
       \frac{\diff}{\diff t} \abs{\tilde m (t)} \leq - \frac{U \sqrt{1 + \alpha^2}}{2}<0,
    \end{equation*}
as long as $\abs{\tilde{m} (t)} < \delta_U$. This yields that, for such time intervals, the mapping $t\mapsto \abs{\tilde m (t)}$ is decreasing. 

Therefore, this shows that if $\delta_U$ satisfies \eqref{choice:deltaU} and $m(0)$ is such that $|m(0)+e_1|<\delta_U$, then $\abs{\tilde{m} (t)} < \delta_U$ for all $t \geq 0$ and that $\tilde m (t)$ reaches $0$ in finite time.
 In other words, $- e_1$ can be reached in finite time with a control $u$ such that $\abs{u} \leq U$ if $m$ is such that $\abs{m + e_1} < \delta_U$. 
 
    To conclude, it remains to drop the dependency of $\delta$ in $U$. Let us use that 
    $- e_1$ is asymptotically stable according to Proposition \ref{prop:asympt_stab}.
    Therefore, there exists $\delta>0$ such that, starting from a point $m(0)$ chosen so that $\abs{m (0) + e_1} < \delta$, we can first let the system evolve without control until we obtain $\abs{m (T_U) + e_1} < \delta_U$ for some finite time $T_U$. From this moment, we know we can reach $- e_1$ in finite time, whence the expected conclusion.
\end{proof}

Recall for the sake of readability that the notation $\mathcal{T}_U$ has been introduced in Section~\ref{sec:OCP}.
\begin{lemma} \label{lem:U_minus_eps}
    Let $\gamma_1 < \gamma_2$ and $U > 0$ such that $\mathcal{T}_U < \infty$. Then there exists $\varepsilon > 0$ such that $\mathcal{T}_{U-\varepsilon} < \infty$.
\end{lemma}

\begin{proof}
    Since $\mathcal{T}_U < \infty$ and according to Theorem~\ref{th:ex_min}, there exists $u_* \in L^\infty (0, \mathcal{T}_U)$ such that $m_* (0) = e_1$ and $m_* (\mathcal{T}_U) = - e_1$. Now, let us consider $m$ the solution to \eqref{LL:ODE} associated to the control choice $u = \frac{U - \varepsilon}{U} u_*$ for some $\varepsilon \in (0, U)$ to be defined later. From Lemma \ref{lem:cont_flow}, we obtain 
    \begin{equation*}
        \norm{m  - m_*}_{W^{1, \infty} (0, \mathcal{T}_U)} \leq C \norm{\frac{U - \varepsilon}{U} u_* - u_*}_{L^\infty (0, \mathcal{T}_U)} = C \varepsilon \frac{\norm{u_*}_{L^\infty (0, \mathcal{T}_U)}}{U} \leq C \varepsilon
    \end{equation*}
    for some $C > 0$.
    
    Since $m_* (\mathcal{T}_U) = - e_1$ by definition, we can take $\varepsilon > 0$ small enough so that $\abs{m (\mathcal{T}_U) + e_1} < \delta$, where $\delta > 0$ is given by Lemma \ref{lem:reach_fin_time}. From this lemma, we know we can reach $- e_1$ in finite time, and since $\abs{u} \leq U - \varepsilon$, this leads to $\mathcal{T}_{U-\varepsilon} < \infty$.
\end{proof}

\begin{lemma} \label{lem:T_U_decreas}
    $\mathcal{T}_U$ is non-increasing with respect to $U > 0$. In particular, if $\mathcal{T}_{U_0} < \infty$ for some $U_0 > 0$, then $\mathcal{T}_U < \infty$ for all $U > U_0$.
\end{lemma}

\begin{proof}
    This property is an immediate consequence of the definition of $\mathcal{T}_U$ and the fact that the sets $\mathcal{O}_{U }$ are increasing for the inclusion.
\end{proof}
\subsection{Emergence of a threshold}
The following result is the most crucial for concluding. It quantifies the asymptotic stability of $e_1$ for the evolution of the magnetization $m$, with respect to $u$ viewed as a perturbation. It is notable that its proof not only highlights the emergence of a threshold but also provides an explicit expression. 
\begin{lemma} \label{lem:U_stab}
Let us assume that $\gamma_1 < \gamma_2$. There exists $U_{\textnormal{stab}}>0$ depending only on $\gamma_3-\gamma_1$, $\gamma_2-\gamma_1$ and $\alpha$ such that, for any $U  < U_{\textnormal{stab}}$, the following holds. Let $(m,u)$ be a solution of \eqref{LL:ODE} on $[0,+\infty)$, such that $u \in L^\infty ([0, \infty))$ and $\norm{u}_{L^\infty} \leq U $. Then for all $t \ge 0$, $m_1(t) \ge 0$.
\end{lemma}

In other words, $m$ remains in the hemisphere with pole $e_1$, and in particular, $m$ can not reach $-e_1$. One has the same statement if $(m,u)$ are defined on a bounded interval $[0,T]$.

\begin{proof}
    Let $v = \alpha (u - (u \cdot m) m) - m \wedge u$. Then, using that $v$ reads as the sum of two orthogonal terms, one has $\abs{v}^2 \leq (1 + \alpha^2) U ^2$.
    Moreover, since $\abs{m}^2 = 1$, there holds
    \begin{equation*}
        Dm \cdot m = \gamma_1 + (\gamma_2 - \gamma_1) m_2^2 + (\gamma_3 - \gamma_1) m_3^2.
    \end{equation*}
As in the proof of Lemma \ref{lem:reach_fin_time}, $\widetilde{m}$ solves the controlled system 
    \begin{equation} \label{eq:ode_tilde_m}
        \dot{\widetilde{m}} = A \widetilde{m} + \xi + \tilde v.
    \end{equation}
where
    \begin{equation} \label{def:A_xi}
        A = 
            \begin{bmatrix}
                - \alpha (\gamma_2 - \gamma_1) & - (\gamma_3 - \gamma_1) \\
                \gamma_2 - \gamma_1 & - \alpha (\gamma_3 - \gamma_1)
            \end{bmatrix}, \quad
        \xi = 
            \begin{bmatrix}
                \alpha ((\gamma_2 - \gamma_1) m_2^2 + (\gamma_3 - \gamma_1) m_3^2) m_2 + (\gamma_3 - \gamma_1) (1 - m_1) m_3 \\
                \alpha ((\gamma_2 - \gamma_1) m_2^2 + (\gamma_3 - \gamma_1) m_3^2) m_3 - (\gamma_2 - \gamma_1) (1 - m_1) m_2
            \end{bmatrix}.
    \end{equation}
    Pay attention to the sign change between $A$ and $\xi$ used here and $A_-$ and $\xi_-$ introduced in the proof of Lemma \ref{lem:reach_fin_time}. 
    %
    %
    Let $\nu \in (0, 1]$ to be fixed later and define
    \[
    T_\nu = \inf \{ t \geq 0 \ |\ |\widetilde{m}(t)| \geq \nu \}.
    \]
    Our goal is to derive suitable bounds on $\tilde m$, so that for a well chosen $\nu$, $T_\nu = +\infty$.
    
    Since $m(0) = e_1$ and $m$ is continuous, we know that $T_\nu > 0$. Note that one has necessarily $m_1(\cdot)>0$ on $(0,T_\nu)$. Then, for all $t \in [0, T_\nu)$, using that $m$ is normalized, there holds like previously
  $
        0 \leq 1 - m_1 (t)\leq 1 - m_1 (t)^2 = \abs{\widetilde{m} (t)}^2,
   $
    and therefore
    \begin{equation*}
        \abs{\xi (t)} \leq (1 + \abs{\alpha}) \, \delta \gamma_+ \, \abs{\widetilde{m} (t)}^3 \leq (1 + \abs{\alpha}) \, \delta \gamma_+ \, \nu^3,
    \end{equation*}
    where $\delta \gamma_+ \coloneqq \gamma_3 - \gamma_1 \geq \gamma_2 - \gamma_1 \eqqcolon \delta \gamma_- > 0$.
    On the other hand, thanks to the Duhamel formula on \eqref{eq:ode_tilde_m} using the fact that $\widetilde{m} (0) = 0$, there holds
    \begin{equation*}
        \widetilde{m} (t) = \int_0^t \exp{\left( (t-s) A \right)} (\xi (s) + \tilde v(s)) \diff s
    \end{equation*}
     for all $t \geq 0$. This, together with the previous estimates, drives to
    \begin{align}
        \abs{\widetilde{m} (t)} &\leq \int_0^t \norm{\exp{\left( (t-s) A \right)}}_2 \left( (1 + \abs{\alpha}) \delta \gamma_+ \nu^3 + \sqrt{1 + \alpha^2} U  \right) \diff s \nonumber \\
            &\leq (1 + \abs{\alpha}) ( \delta \gamma_+ \nu^3 + U  ) \int_0^t \norm{\exp{\left( (t-s) A \right)}}_2 \diff s \nonumber \\
            &\leq (1 + \abs{\alpha}) ( \delta \gamma_+ \nu^3 + U  ) \int_0^{t} \norm{\exp{\left( s A \right)}}_2 \diff s,\label{m:1903} 
    \end{align}
    for all $t \in [0, T_\nu)$, where $\norm{\cdot}_2$ still denotes the operator norm associated to the euclidean norm $\abs{\cdot}$.

    We will now provide an estimate of the norm of the exponential matrix. Recall that the characteristic polynomial of $A$ is $P_A (X) = X^2 - \Tr (A) X + \det(A)$ with
    \begin{equation*}
        \det(A) = (1 + \alpha^2) \delta \gamma_- \, \delta \gamma_+ > 0, \qquad
        \Tr (A) = - \alpha (\delta \gamma_- + \delta \gamma_+) < 0.
    \end{equation*}
    Its discriminant $\Delta$ reads $\Delta = \Tr (A)^2 - 4 \det(A) = \alpha^2 (\delta \gamma_+ - \delta \gamma_-)^2 - 4\delta \gamma_- \delta \gamma_+$. To compute the eigenvalues of $A$, we have to distinguish between several cases.

\paragraph{1st case: $\Delta > 0$.}
    Then its eigenvalues are $\lambda_\pm \coloneqq \frac12(\Tr (A) \pm \sqrt{\Delta})$.
    Remark that both eigenvalues of $A$ are negative (according to the signs of the trace and the determinant above) and different from each other, which means that $A$ is diagonalizable. 
    Therefore, we infer\footnote{Here, the Lagrange interpolation formula is used to compute the exponential of $A$: for every matrix $M\in \mathcal{M}_d(\C)$ whose spectrum $\{\lambda_i\}_{1\leq i\leq d}$ consists of distinct eigenvalues, one has
    $$
    \exp(M)=\sum_{j=1}^de^{\lambda_j}\prod_{i\neq j}\frac{M-\lambda_i\operatorname{I}_d}{\lambda_j-\lambda_i} .
    $$
    } that
    \begin{align*}
        \exp{( sA )} &= e^{s \lambda_+} \frac{sA - s \lambda_- \operatorname{I}_2}{s \lambda_+ - s \lambda_-} + e^{s \lambda_-} \frac{sA - s \lambda_+ \operatorname{I}_2}{s \lambda_- - s \lambda_+} = \frac{1}{\sqrt{\Delta}} \left( e^{s \lambda_+} (A - \lambda_- \operatorname{I}_2) - e^{s \lambda_-} (A - \lambda_+ \operatorname{I}_2) \right) \\
            &= \frac{1}{\sqrt{\Delta}} \left( (e^{s \lambda_+} - e^{s \lambda_-}) A + (e^{s \lambda_-} \lambda_+ - e^{s \lambda_+} \lambda_-) \operatorname{I}_2 \right) \\
            &= \frac{e^{s \lambda_+}}{\sqrt{\Delta}} \left( (1 - e^{- s \sqrt{\Delta}}) A + (e^{- s \sqrt{\Delta}} \lambda_+ - \lambda_-) \operatorname{I}_2 \right) \\
            &= s e^{s \lambda_+} \biggl( \frac{1 - e^{- s \sqrt{\Delta}}}{s \sqrt{\Delta}} A - \lambda_- \frac{1 - e^{- s \sqrt{\Delta}}}{s \sqrt{\Delta}}  \operatorname{I}_2 \biggr) + e^{s \lambda_-} \operatorname{I}_2.
    \end{align*}
    Thus, using the facts that $\lambda_- < \lambda_+ < 0$, $\norm{A}_2 \geq \abs{\lambda_-}$ and also that the function $f$ given by $f(x) = \frac{1 - e^{-x}}{x}$ analytically extended to $\R$ is uniformly bounded by 1 on $[0, \infty)$, we get
    \begin{equation*}
        \norm{\exp{( sA )}}_2 \leq e^{s \lambda_+} \left( s (\norm{A}_2 + \abs{\lambda_-}) + 1 \right) \leq e^{s \lambda_+} \left( 2 s \norm{A}_2 + 1 \right).
    \end{equation*}
    Hence,
    \begin{align*}
        \int_0^{t} \norm{\exp{\left( s A \right)}}_2 \diff s &\leq \abs{\lambda_+}^{-1} (1 - e^{\lambda_+ t}) + 2 \norm{A}_2 \abs{\lambda_+}^{-2} (1 - (|\lambda_+| t + 1) e^{\lambda_+ t}) \\
            &\leq (1 - e^{\lambda_+ t}) \left( \abs{\lambda_+}^{-1} + 2 \norm{A}_2 \abs{\lambda_+}^{-2} \right) \\
            &\leq 3 (1 - e^{\lambda_+ t}) \norm{A}_2 \abs{\lambda_+}^{-2},
    \end{align*}
    and according to \eqref{m:1903}, one has for all $t \in [0, T_\nu)$
    \begin{equation*}
        \abs{\widetilde{m} (t)} \leq 3 \norm{A}_2 \abs{\lambda_+}^{-2} (1 - e^{\lambda_+ t}) (1 + \abs{\alpha}) ( \delta \gamma_+ \nu^3 + U  )
    \end{equation*}
    To conclude, we will choose $U $ adequately so that the function $x \mapsto 3 \norm{A}_2 \abs{\lambda_+}^{-2} (1 + \abs{\alpha}) ( \delta \gamma_+ x^3 + U  )$ admits a fixed point $x_0$ in $(0, 1]$. 
    This is possible thanks to the next lemma, whose proof is postponed to the end of this section for the sake of clarity.
    \begin{lemma} \label{lem:fixed_point}
        Let $a, b, c > 0$. The function $x \mapsto a^{-1} (b x^3 + c)$ has a fixed point $x_0$ in $(0, 1]$ if, and only if $c \leq a x_1 - b x_1^3$ where $x_1 = \min \{1, \sqrt{\frac{a}{3 b}} \}$.
    \end{lemma}
    Remark that if $x_1$ is as in this lemma, one has $a x_1 - b x_1^3 \geq \frac{2}{3} a x_1>0$.
    Hence, setting $a=|\lambda_+|^2/(3\norm{A}_2(1+|\alpha|))$, $b=\delta \gamma_+$ and $c=U $ drives us to 
    assume that 
    \begin{equation*}
        U  \leq \frac{\abs{\lambda_+}^2}{3 \norm{A}_2 (1 + \abs{\alpha})} x_1 - \delta \gamma_+ \, x_1^3, \quad\text{with}\quad x_1 \coloneqq \min \left\{ 1, \frac{\abs{\lambda_+}}{3 \sqrt{\norm{A}_2 (1 + \abs{\alpha}) \delta \gamma_+}} \right\},
    \end{equation*}
    we can take $\nu = x_0$ provided by Lemma~\ref{lem:fixed_point}, and the previous estimate leads to
    \begin{equation*}
        \abs{\widetilde{m}(t)} \leq (1 - e^{\lambda_+ t}) \nu,
    \end{equation*}
    for all $t \in [0, T_\nu)$. A continuity argument then implies that $T_\nu = \infty$. In other words, for all $t \geq 0$, $|m_1(t)| = \sqrt{1-|\tilde m(t)|^2} \ge \sqrt{1-(1 - e^{\lambda_+ t})^2 \nu^2} >0$.
    Now, $m_1$ is continuous, so that it keeps a constant sign. As $m_1(0)=1$, $m_1(t) \ge 0$ for all $t \ge 0$, which is the desired conclusion.
    \medskip
    
    \paragraph{2nd case: $\Delta < 0$.}
    In this case, the eigenvalues are
    \begin{equation*}
        \lambda_\pm \coloneqq \frac{\Tr (A) \pm i \sqrt{- \Delta}}{2}.
    \end{equation*}
    One more time, the two eigenvalues are distinct, complex conjugate with negative real part. Yet, the same decompositions as previously can still be applied and there holds
    \begin{align*}
        \exp{( sA )} &= e^{s \lambda_+} \frac{sA - s \lambda_- \operatorname{I}_2}{s \lambda_+ - s \lambda_-} + e^{s \lambda_-} \frac{sA - s \lambda_+ \operatorname{I}_2}{s \lambda_- - s \lambda_+} \\
            &= \frac{1}{i \sqrt{- \Delta}} \left( e^{s \lambda_+} (A - \lambda_- \operatorname{I}_2) - e^{s \lambda_-} (A - \lambda_+ \operatorname{I}_2) \right) \\
            &= \frac{1}{i \sqrt{- \Delta}} \left( (e^{s \lambda_+} - e^{s \lambda_-}) A + (e^{s \lambda_-} \lambda_+ - e^{s \lambda_+} \lambda_-) \operatorname{I}_2 \right) \\
            &= \frac{e^{\frac{s}{2} \Tr (A)}}{i \sqrt{- \Delta}} \left( (e^{\frac{i s \sqrt{- \Delta}}{2}} - e^{- \frac{i s \sqrt{- \Delta}}{2}}) A + (\lambda_+ e^{- \frac{i s \sqrt{- \Delta}}{2}} - \lambda_- e^{\frac{i s \sqrt{- \Delta}}{2}}) \operatorname{I}_2 \right) \\
            &= \frac{e^{\frac{s}{2} \Tr (A)}}{\sqrt{- \Delta}} \Bigl[ 2 \sin{\left( \frac{s \sqrt{- \Delta}}{2} \right)} A + \left( \sqrt{- \Delta} \cos{\left( \frac{s \sqrt{- \Delta}}{2} \right)} - \Tr (A) \sin{\left( \frac{s \sqrt{- \Delta}}{2} \right)} \right) \operatorname{I}_2 \Bigr] \\
            &= \frac{s}{2} e^{\frac{s}{2} \Tr (A)} \Bigl[ 2 A - \Tr (A) \operatorname{I}_2 \Bigr] \sinc{\left( \frac{s \sqrt{- \Delta}}{2} \right)} + e^{\frac{s}{2} \Tr (A)} \cos{\left( \frac{s \sqrt{- \Delta}}{2} \right)} \operatorname{I}_2.
    \end{align*}
    Thus, we get
    \begin{equation*}
        \norm{\exp{( sA )}}_2 \leq \frac{s}{2} e^{\frac{s}{2} \Tr (A)} \left(2 \norm{A}_2 - \Tr (A) \right) + e^{\frac{s}{2} \Tr (A)} \leq e^{\frac{s}{2} \Tr (A)} \left( 2 s \norm{A}_2 + 1 \right),
    \end{equation*}
    since $\Tr (A) = \lambda_+ + \lambda_-$ and $\abs{\lambda_\pm} \leq \norm{A}_2$.
    Hence, following the same way as in the first case, we get
    \begin{align*}
        \int_0^{t} \norm{\exp{\left( s A \right)}}_2 \diff s &\leq (1 - e^{\frac{1}{2} \Tr (A) t}) \left( 2 \abs{\Tr (A)}^{-1} + 8 \norm{A}_2 \abs{\Tr (A)}^{-2} \right) \\
            &\leq 12 (1 - e^{\frac{1}{2} \Tr (A) t}) \norm{A}_2 \abs{\Tr (A)}^{-2},
    \end{align*}
    and according to \eqref{m:1903}, one has for all $t \in [0, T_\nu)$
    \begin{equation*}
        \abs{\widetilde{m} (t)} \leq 12 \norm{A}_2 \abs{\Tr (A)}^{-2} (1 - e^{\frac{1}{2}\Tr (A) t}) (1 + \abs{\alpha}) ( \delta \gamma_+ \nu^3 + U  )
    \end{equation*}
    Now, by mimicking the reasoning done in the first case, by assuming
    \begin{equation*}
        U  \leq \frac{\Tr (A)^2}{12 \norm{A}_2 (1 + \abs{\alpha})} x_1 - \delta \gamma_+ \, x_1^3, \quad\text{with}\quad  x_1 \coloneqq \min \left\{ 1, \frac{\Tr (A)}{6 \sqrt{\norm{A}_2 (1 + \abs{\alpha}) \delta \gamma_+}} \right\},
    \end{equation*}
    and taking $\nu = x_0$ given by Lemma~\ref{lem:fixed_point}, the previous estimate leads to
    \begin{equation*}
        \abs{\widetilde{m}(t)} \leq (1 - e^{\frac{1}{2} \Tr (A) t}) \nu,
    \end{equation*}
    for all $t \in [0, T_\nu)$. Arguing as in the first case, we infer that $T_\nu = \infty$ in this case as well, and then, $m_1(t) > 0$ for all $t \ge 0$.
    \medskip
    
    \paragraph{3rd case: $\Delta = 0$.}
    In this case, both eigenvalues are equal, one has $\lambda = \Tr (A)/2<0$.
    Note that, in that case, $A - \frac{1}{2} \Tr (A) \operatorname{I}_2$ is therefore a non-zero nilpotent matrix, and more precisely $(A - \frac{1}{2}\Tr (A) \operatorname{I}_2)^2 = (s A - \frac{s}{2} \Tr (A) \operatorname{I}_2)^2 = 0$. Thus, there holds
    \begin{align*}
        \exp{(sA)} &= \exp{\left( \frac{s}{2} \Tr (A) \operatorname{I}_2 \right)} \, \exp{( sA - \frac{s}{2} \Tr (A) \operatorname{I}_2 )} \\
            &= e^{\frac{s}{2} \Tr (A)} (\operatorname{I}_2 + sA - \frac{s}{2} \Tr (A) \operatorname{I}_2)
    \end{align*}
    which yields
    \begin{equation*}
        \norm{\exp{(sA)}}_2 \leq e^{\frac{s}{2} \Tr (A)} (1 + s (\norm{A}_2 - \frac{1}{2} \Tr (A))) \leq e^{\frac{s}{2} \Tr (A)} (1 + 2 s \norm{A}_2).
    \end{equation*}
    The computations are then exactly the same ones as in the second case, and the conclusion follows in the same fashion.
\end{proof}

\begin{proof}[Proof of Lemma \ref{lem:fixed_point}]
    We are looking for a root $x_0 \in (0, 1]$ of the polynomial function $f$ given by $f (X) = b X^3 - a X + c$, whose derivative $3 b X^2 - a$ is negative for $X < \sqrt{\frac{a}{3 b}} \eqqcolon x_2$ and positive for $X > x_2$. The minimum in $[0, 1]$ is therefore reached at $x_1$ and, since $f (0) = c > 0$, there is a root if and only if $f(x_1) \leq 0$, which corresponds to the assumption in the statement.
    %
\end{proof}

\begin{remark} \label{rk:theoremUcrit}
From the proof of Lemma~\ref{lem:U_stab}, we obtained the following expression for $U_{\text{stab}}$. Consider the matrix $A$ defined there in \eqref{def:A_xi}, denote
$\Delta = \Tr (A)^2 - 4 \det(A) $ the discriminant of its characteristic polynomial, $\lambda_\pm$ its eigenvalues chosen so that $\lambda_+>\lambda_-$ whenever $\Delta>0$, and $\delta \gamma_+ \coloneqq \gamma_3 - \gamma_1 > 0$. 

Let
\[
x_1(A) \coloneqq \left\{\begin{array}{ll}
\min \left\{ 1, \frac{\abs{\lambda_+}}{3 \sqrt{\norm{A}_2 (1 + \abs{\alpha}) \delta \gamma_+}} \right\} & \text{if }\Delta >0\\
\min \left\{ 1, \frac{\operatorname{Tr}(A)}{6 \sqrt{\norm{A}_2 (1 + \abs{\alpha}) \delta \gamma_+}} \right\} & \text{else}.
\end{array}\right.
\]
Then 
\[ U_{\rm stab} =
\Gamma(\Delta) \coloneqq \left\{\begin{array}{ll}
\frac{\abs{\lambda_+}^2}{3 \norm{A}_2 (1 + \abs{\alpha})} x_1(A) - \delta \gamma_+ \, x_1(A)^3 & \text{if }\Delta >0\\
\frac{\operatorname{Tr}(A)^2}{12 \norm{A}_2 (1 + \abs{\alpha})} x_1(A) - \delta \gamma_+ \, x_1(A)^3 & \text{else}.
\end{array}\right.
\]
Note also that, to complement this result, explicit computations of the quantities involved (like $\| A \|_2$) are provided in Appendix~\ref{append:comp}.
\end{remark}

We now have all the elements to conclude the:

\begin{proof}[Proof of Theorem \ref{th:ex_U_crit}]

  Define
    \begin{equation*}
        U_\textnormal{crit} \coloneqq \inf \{ U \, | \, \mathcal{T}_U < \infty \}.
    \end{equation*}
    From Lemma \ref{lem:U_stab}, we know that $U_\textnormal{crit} > 0$. Lemma~\ref{lem:T_U_decreas}  proves the second point. To investigate the case where $U =U_\textnormal{crit}$, observe that, by definition, \eqref{OCP:minT} has no solution for all $U \in (0, U_\textnormal{crit})$. Moreover, if \eqref{OCP:minT} had a solution for $U = U_\textnormal{crit}$, then Lemma~\ref{lem:U_minus_eps} would provide a contradiction with respect to the definition of $U_\textnormal{crit}$.
\end{proof}


\section{Cases with symmetry}\label{sec:ProofSym}

In this section, we deal with the two cases when the material satisfies additional symmetry without being a sphere (in which case the analysis becomes trivial). They correspond to the cases $\gamma_1 = \gamma_2 < \gamma_3$ and $\gamma_1 < \gamma_2 = \gamma_3$.

\subsection{Proof of Theorem~\ref{theo:gamma1=gamma2} (case \texorpdfstring{$\gamma_1=\gamma_2$}{gamma1=gamma2})}





From Theorem \ref{th:ex_min}, we have to investigate the existence of an admissible trajectory for this problem, in other words, the existence of a control $u\in \mathcal{O}_U$ and a time $T>0$ such that $m_u(T)=-e_1$. This property is known to be true as soon as $U $ is large enough according to \cite{alouges2009magnetization}. But it has to be proved for smaller $U $.



Let us assume that $\gamma_1=\gamma_2$. 
We will prove that, in that case, $\inf_{T>0}\Lambda(T)=0$, (with $\Lambda(T)$ defined by Equation \eqref{ob:aux}) which will prove that Problem~\eqref{OCP:minT} has a solution whatever the value of $U >0$. For $\varepsilon>0$, let us consider a particular trajectory $m_\varepsilon$ of the form $m_\varepsilon=(\cos (\varepsilon t),\sin(\varepsilon t),0)$. Then, by defining 
$$
\mathcal F_\varepsilon: = \left(\dot e_\varepsilon\cdot (Dm_\varepsilon\wedge m_\varepsilon)+\frac{|\dot m_\varepsilon|}{1+\alpha^2}\right)^2+\left(\dot e_\varepsilon\cdot Dm_\varepsilon+\frac{\alpha|\dot m_\varepsilon|}{1+\alpha^2}\right)^2
$$ 
with $\dot e_\varepsilon=\dot m_\varepsilon/|\dot m_\varepsilon|$, a straightforward computation yields
$$
\mathcal F_\varepsilon=\frac{\varepsilon^2}{1+\alpha^2}\leq \varepsilon^2.
$$
We infer that $\inf_{T>0}\Lambda(T)\leq \mathcal F_\varepsilon\leq \varepsilon^2$ whence the conclusion, since $\varepsilon$ is arbitrary.

Let us now prove the last point of this result, assuming that from now on $\gamma_1<\gamma_3$. 
Assume that $m_3 (t) = 0$ for all $t \geq 0$. Then, $D m = \gamma_1 m$. By contradiction, if such an $m$ is an optimal trajectory, Proposition \ref{prop:1stOrderCond} is satisfied, and \eqref{eq:LL_ODE_bis} gives
\begin{equation*}
    \dot{m} = U  (\alpha^2 + 1)^{1/2} \frac{\varphi}{\abs{\varphi}}.
\end{equation*}
By taking the third coordinate, we get $\varphi_3 (t) = 0$ for all $t \geq 0$. Thus, we also get $D \varphi = \gamma_1 \varphi$ and \eqref{eq:complete_eq_phi} gives
\begin{equation*}
    \dot{\varphi} = - \gamma_1 m \wedge \varphi - D (\varphi \wedge m) - U  (\alpha^2 + 1)^{1/2} \abs{\varphi} m.
\end{equation*}
By taking again the third coordinate, we get
\begin{equation*}
    0 = - \gamma_1 (m \wedge \varphi) \cdot e_3 - D (\varphi \wedge m) \cdot e_3 = (\gamma_3 - \gamma_1) (m \wedge \varphi) \cdot e_3 = (\gamma_3 - \gamma_1) (e_3 \wedge m) \cdot \varphi.
\end{equation*}
Since $\gamma_3 > \gamma_1$, this proves that $(e_3 \wedge m) \cdot \varphi = 0$. However, at $t=0$, this means that
\begin{equation*}
    0 = (e_3 \wedge e_1) \varphi (0) = \varphi_2 (0).
\end{equation*}
Now $\varphi_1 (0) = \varphi (0) \cdot m (0) = 0$, and we obtained $\varphi (0) = 0$: this is  a contradiction with \eqref{cond:phi:timeTtime0}.



\subsection{Proof of Theorem~\ref{theo:gamma2=gamma3} (case \texorpdfstring{$\gamma_2=\gamma_3$}{gamma2=gamma3})}
For this case, we first show that the (PMP) conditions are also sufficient conditions for optimal trajectories :

\begin{lemma} \label{lem:link_pmp_opt_traj_23}
    Let $U > U_\textnormal{crit}$. Then any trajectory $m$ satisfying the (PMP) (\eqref{eq:adj_eq}-\eqref{HamiltonCst} with $p^0 = -1$) is an optimal trajectory.
\end{lemma}

\begin{proof}
    Let $m_*$ be an optimal trajectory, and $p_*$ the associated adjoint state. By definition, they satisfy the (PMP) conditions. Now, let $(m, p)$ be a trajectory and its adjoint state satisfying the (PMP) conditions. Let also $\varphi = p - (p \cdot m) m$ and $\varphi_* = p_* - (p_* \cdot m_*) m_*$. In particular, we know that $\varphi (0)$ satisfies \eqref{cond:phi:timeTtime0} and $\varphi (0) \perp m (0) = e_1$, and similarly for $\varphi_*$ with respect to $m_*$. Thus, there exists $\theta \in [0, 2 \pi]$ such that $R_\theta \varphi (0) = \varphi_* (0)$ where $R_\theta$ is the rotation along $e_1$ of angle $\theta$:
    \begin{equation*}
        R_\theta =
            \begin{pmatrix}
                1 & 0 & 0 \\
                0 & \cos \theta & - \sin \theta \\
                0 & \sin \theta & \cos \theta
            \end{pmatrix}.
    \end{equation*}
    On the other hand, since $\gamma_2 = \gamma_3$, we have $D R_\theta = R_\theta D$ for all $\theta$, but also $R_\theta f \wedge R_\theta g = R_\theta (f \wedge g)$ for any $f, g \in \mathbb{R}^3$. Last, $R_\theta m_* (0) = R_\theta e_1 = e_1$. Thus, $(R_\theta m, R_\theta \varphi)$ satisfies the same system of ODEs as $(m_*, \varphi_*)$ (i.e. \eqref{LL:ODE}-\eqref{eq:adj_eq} with $u (\cdot)$ or $u_* (\cdot)$ satisfying \eqref{expr:uopt1809}) with the same initial data. By the Cauchy-Lipschtiz theorem and using the fact that both $\varphi$ and $\varphi_*$ never vanish thanks to \eqref{eq:expr_max_ham}, we obtain $(R_\theta m, R_\theta \varphi) = (m_*, \varphi_*)$, i.e. $(m, \varphi) = (R_{- \theta} m_*, R_{- \theta} \varphi_*)$, and thus the conclusion.
\end{proof}

The following two results exploit in a precise way the (necessary and sufficient) optimality conditions.

\begin{lemma} \label{lem:adjoint_dir}
    Let $U > 0$ and $(m, p)$ satisfy the (PMP) conditions (\eqref{eq:adj_eq}-\eqref{HamiltonCst} with $p^0 = -1$) and $\varphi = p - (p \cdot m) m$. Then, for every $t \geq 0$, $\varphi (t) \cdot (e_1 \wedge m (t)) = 0$.
\end{lemma}

\begin{proof}
    We know that $\varphi$ satisfies \eqref{eq:complete_eq_phi} and $m$ satisfies \eqref{eq:LL_ODE_bis} with $u$ given by \eqref{expr:uopt1809}. Moreover,
    \begin{equation*}
        \frac{\diff}{\diff t} (\varphi \cdot (e_1 \wedge m)) = \dot{\varphi} \cdot (e_1 \wedge m) + \varphi \cdot (e_1 \wedge \dot{m}).
    \end{equation*}
    Using the facts that $u \perp m$ and $m \perp (e_1 \wedge m)$, there holds
    \begin{multline*}
        \dot{\varphi} \cdot (e_1 \wedge m) = \alpha D \varphi \cdot (e_1 \wedge m) - \alpha (Dm \cdot m) \varphi \cdot (e_1 \wedge m) - (Dm \wedge \varphi) \cdot (e_1 \wedge m) - D (\varphi \wedge m) \cdot (e_1 \wedge m) \\ + \frac{U}{\sqrt{1 + \alpha^2}} (\varphi \wedge (\frac{\varphi}{|\varphi|} \wedge m)) \cdot (e_1 \wedge m),
    \end{multline*}
    \begin{equation*}
        \varphi \cdot (e_1 \wedge \dot{m}) = - \alpha \varphi \cdot (e_1 \wedge Dm) + \alpha (Dm \cdot m) \varphi \cdot (e_1 \wedge m) + \varphi \cdot (e_1 \wedge (m \wedge Dm)) + U \sqrt{1 + \alpha^2} \varphi \cdot (e_1 \wedge \frac{\varphi}{\abs{\varphi}}).
    \end{equation*}
    First, we point out that $\varphi \perp m$ so that $\varphi \wedge (\varphi \wedge m) = - \abs{\varphi}^2 m$, and thus $(\varphi \wedge (\varphi \wedge m)) \cdot (e_1 \wedge m) = 0$. Similarly, $U \sqrt{1 + \alpha^2}  \varphi \cdot (e_1 \wedge \frac{\varphi}{\abs{\varphi}}) = 0$.
    Then, using the triple product formula, we get
    \begin{gather*}
        \varphi \cdot (e_1 \wedge Dm) = - Dm \cdot (e_1 \wedge \varphi)
\\
        (Dm \wedge \varphi) \cdot (e_1 \wedge m) = (\varphi \wedge (e_1 \wedge m)) \cdot Dm = - (\varphi \cdot e_1) (m \cdot Dm),
\\
        (m \wedge Dm) \cdot (e_1 \wedge \varphi) = ((e_1 \wedge \varphi) \wedge m) \cdot Dm = (m \cdot e_1) (\varphi \cdot Dm).
    \end{gather*}
    Moreover, since $\gamma_2 = \gamma_3$, we know that, for any vector $f \in \mathbb{R}^3$ such that $f \cdot e_1 = 0$, $D f = \gamma_2 f = \gamma_3 f$. With the fact that $D$ is symmetric, this leads to
    \begin{gather*}
        D \varphi \cdot (e_1 \wedge m) = \varphi \cdot D (e_1 \wedge m) = \gamma_2 \varphi \cdot (e_1 \wedge m),
\\
       Dm \cdot (e_1 \wedge \varphi) = m \cdot D(e_1 \wedge \varphi) = \gamma_2 m \cdot (e_1 \wedge \varphi) = - \gamma_2 \varphi \cdot (e_1 \wedge m), \\
        D (\varphi \wedge m) \cdot (e_1 \wedge m) = (\varphi \wedge m) \cdot D (e_1 \wedge m) = \gamma_2 (\varphi \wedge m) \cdot (e_1 \wedge m).
    \end{gather*}
    Last, using again the double product, we have
    \begin{equation*}
        \varphi \cdot (e_1 \wedge (m \wedge Dm)) = (\varphi \cdot m) (e_1 \cdot Dm) - (\varphi \cdot Dm) (e_1 \cdot m) = - (\varphi \cdot Dm) (e_1 \cdot m).
    \end{equation*}
    But one has then
    \begin{align*}
        (\varphi \cdot e_1) (m \cdot Dm) - (\varphi \cdot Dm) (e_1 \cdot m) &= e_1 \cdot \Bigl[ (Dm \cdot m) \varphi - (Dm \cdot \varphi) m \Bigr] = -e_1 \cdot (Dm \wedge (m \wedge \varphi)) \\
            &= -Dm \cdot ((m \wedge \varphi) \wedge e_1) = -m \cdot D ((m \wedge \varphi) \wedge e_1) \\
            &= -\gamma_2 m \cdot ((m \wedge \varphi) \wedge e_1) =  \gamma_2 (\varphi \wedge m) \cdot (e_1 \wedge m).
    \end{align*}
    This means
    \begin{equation*}
        \varphi \cdot (e_1 \wedge (m \wedge Dm)) - (Dm \wedge \varphi) \cdot (e_1 \wedge m) - D (\varphi \wedge m) \cdot (e_1 \wedge m) = 0.
    \end{equation*}
    Hence,
    \begin{equation*}
        \frac{\diff}{\diff t} (\varphi \cdot (e_1 \wedge m)) = 0.
    \end{equation*}
    The conclusion comes by integration, noticing furthermore that $m(0) = e_1$ and thus $e_1 \wedge m(0) = 0$.
\end{proof}

\begin{lemma} \label{lem:m1_theta_odes}
    Let $U > 0$ and $(m, p)$ satisfy the (PMP) conditions (\eqref{eq:adj_eq}-\eqref{HamiltonCst} with $p^0 = -1$). Denote $m=(m_1,m_2,m_3)$ the coordinates of $m$, and define $t_0 \coloneqq \inf \{ t \geq 0 \, | \, m (t) = \pm e_1 \} > 0$ (possibly $+ \infty$) and $\theta \in [0, \pi]$ such that $m_1 = \cos \theta$ on $[0, t_0)$. Then $m_1$ and $\theta$ satisfy on $[0, t_0)$
    \begin{equation*}
        \dot{m}_1 = \alpha (\gamma_2 - \gamma_1) (1 - m_1^2) m_1 - U \sqrt{1 + \alpha^2} \sqrt{1 - m_1^2},
    \end{equation*}
    \begin{equation} \label{eq:theta_ode}
        \dot{\theta} = - \alpha (\gamma_2 - \gamma_1) \sin{\theta} \, \cos{\theta} + U \sqrt{1 + \alpha^2}.
    \end{equation}
    Last, $t_0 = \infty$ if $U \leq U_\textnormal{crit}$ and $t_0 = \mathcal{T}_U$ if $U > U_\textnormal{crit}$
\end{lemma}

\begin{proof}
    Let $p$ its adjoint state and $\varphi = p - (p \cdot m) m$. Then $\varphi \cdot (e_1 \wedge m) = 0$ from Lemma \ref{lem:adjoint_dir}, which means that $\varphi$ is orthogonal to both $m$ and $e_1 \wedge m$ for all times in $[0, \mathcal{T}_U]$. Moreover, as soon as $m(t) \neq \pm e_1$ (i.e. as soon as $t \in (0, \mathcal{T}_U)$), $(m (t), e_1 \wedge m(t), m(t) \wedge (e_1 \wedge m(t)))$ is an orthogonal basis of $\mathbb{R}^3$. Therefore, $\varphi (t)$ is colinear to $m(t) \wedge (e_1 \wedge m(t))$, i.e. there is $\lambda \in \mathscr C((0,\mathcal{T}_U),\mathbb R)$ such that
    \begin{equation*}
        \varphi = \lambda m \wedge (e_1 \wedge m) = \lambda \Bigl[ e_1 - m_1 m \Bigr].
    \end{equation*}
    Moreover, from \eqref{eq:H_const}, we know that $\varphi$ does not vanish, thus neither does $\lambda$, which has a constant sign. Then, we also have
    \begin{equation*}
        e_1 \cdot \varphi (t) = \lambda (t) (1 - m_1^2)
  \quad \text{and}\quad 
        \abs{\varphi (t)} = \abs{\lambda (t)} \sqrt{1 - m_1^2},
    \end{equation*}
    which leads to
    \begin{equation*}
        e_1 \cdot \frac{\varphi (t)}{\abs{\varphi (t)}} = \operatorname{sign} (\lambda) \sqrt{1 - m_1^2},
    \end{equation*}
    with $\operatorname{sign} (\lambda) = \pm 1$ constant in time. 
    We also have
    \begin{equation*}
        (Dm \cdot m) = \gamma_1 m_1^2 + \gamma_2 (m_2^2 + m_3^2) = \gamma_2 - (\gamma_2 - \gamma_1) m_1^2,
    \end{equation*}
    \begin{equation*}
        e_1 \cdot (m \wedge Dm) = Dm \cdot (e_1 \wedge m) = m \cdot D (e_1 \wedge m) = m \cdot \gamma_2 (e_1 \wedge m) = 0.
    \end{equation*}
    Therefore, the evolution equation on $m_1$ is
    \begin{equation} \label{eq:dot_m1_interm}
        \dot{m}_1 = \alpha (\gamma_2 - \gamma_1) (1 - m_1^2) \, m_1 + \operatorname{sign} (\lambda) U \sqrt{1 + \alpha^2} \sqrt{1 - m_1^2}.
    \end{equation}
    On the other hand, we know that $\tilde{m} = (m_2, m_3)$ satisfies at $t = 0$:
    \begin{equation*}
        \dot{\tilde m} (0) = U \sqrt{1 + \alpha^2} \frac{\tilde{\varphi} (0)}{\abs{\varphi (0)}},
    \end{equation*}
    with $\tilde{\varphi} (0) \neq 0$ since $\varphi (0) \cdot e_1 = 0$ and $\abs{\varphi (0)} = \frac{1}{U \sqrt{1 + \alpha^2}} > 0$. Since $\tilde m (0) = (0, 0)$, this means that $\abs{\tilde m}$ is not vanishing on $(0, \varepsilon]$ for some $\varepsilon > 0$ small enough. Since $\abs{m}^2 = 1$, this necessarily means that $m_1^2 < 1$ on $(0, \varepsilon]$.
    Now, we can introduce $\theta (t)$ the first angle of the spherical coordinate such that $m_1 = \cos \theta$, and we can assume that $\theta (0) = 0$ and $\theta (t) > 0$ on $(0, \varepsilon]$. The angle $\theta (t)$ is then well defined on $[0, t_0)$ where $t_0 = \min \{ t > 0 \, | \, m(t) = \pm e_1 \}$ and $\theta (t) \in [0, \pi]$. Moreover, since $m_1$ is $\mathscr{C}^1$ (due to \eqref{eq:dot_m1_interm}, for example), $\theta$ is $\mathscr{C}^1$ on $(0, t_0)$ as well.
    %
    Then, on this interval, we can replace $m_1$ in \eqref{eq:dot_m1_interm} by its expression in terms of $\theta$, which leads to
    \begin{equation*}
        - \dot{\theta} \sin{\theta} = \alpha (\gamma_2 - \gamma_1) \sin^2 \theta \, \cos{\theta} + \operatorname{sign} (\lambda) U \sqrt{1 + \alpha^2} \sin{\theta},
    \end{equation*}
    hence
    \begin{equation*}
        \dot{\theta} = - \alpha (\gamma_2 - \gamma_1) \sin{\theta} \cos{\theta} - \operatorname{sign} (\lambda) U \sqrt{1 + \alpha^2}.
    \end{equation*}
    From this, we also see that $\theta$ is $\mathscr{C}^1$ at $t = 0$ with $\dot{\theta} (0) = - \operatorname{sign} (\lambda) U \sqrt{1 + \alpha^2}$.
    As $\theta \geq 0$ on $[0, t_0]$, we can easily see that $\operatorname{sign} (\lambda) = -1$ (otherwise we would have $\dot{\theta} (0) < 0$). This gives the expected ODEs on $m_1$ and $\theta$, but on $[0, t_0)$ only.
    
    To conclude, we shall prove that $t_0 = \infty$ if $U \leq U_\textnormal{crit}$ or $t_0 = \mathcal{T}_U$ if $U > U_\textnormal{crit}$, which is equivalent to prove that $m$ does not reach $e_1$ again (up to reaching $-e_1$ before), or equivalently that $\theta$ does not come back to $0$ before reaching $\pi$. This follows from the fact that $\theta$ satisfies an autonomous first-order ODE of the form $\dot \theta = f(\theta)$ with $f(0) > 0$.
 \end{proof}
We are now in position to prove Theorem~\ref{theo:gamma2=gamma3}.

\begin{proof}[Proof of Theorem~\ref{theo:gamma2=gamma3}]
  Let $U > 0$ and $(m, \varphi)$ satisfying the (PMP) conditions (\eqref{eq:adj_eq}-\eqref{HamiltonCst} with $p^0 = -1$). From Lemma \ref{lem:link_pmp_opt_traj_23} and Theorem \ref{th:ex_U_crit}, we have 2 cases:
    \begin{itemize}
        \item either $U \leq U_\textnormal{crit}$, and then no trajectory reaches $- e_1$ (and so in particular $m$).
        \item either $U > U_\textnormal{crit}$, and then $m$ reaches $- e_1$.
    \end{itemize}
    Therefore, we shall analyze only the case when $U > U_{\rm crit}$ and $m$ is able to reach $- e_1$ and. From Lemma \ref{lem:m1_theta_odes}, we can define $\theta (t) \in [0, \pi]$ such that $m_1 = \cos \theta$, and it satisfies \eqref{eq:theta_ode}. Since it is an autonomous ODE of the form $\dot \theta = f(\theta)$ with $f(0) > 0$, it is easy to prove that $\theta$ is able to reach $\pi$ (which means $m_1$ reaches $-1$ or also $m$ reaches $- e_1$) if and only if $f > 0$ on $[0, \pi]$, where $f (x) = - \alpha (\gamma_2 - \gamma_1) \sin{x} \cos{x} + U \sqrt{1 + \alpha^2}$.
    From this,
    \begin{align*}
        f (x) > 0 \quad \forall x \in [0, \pi] \quad &\Longleftrightarrow \quad \frac{1}{2} \sin{(2x)} < \frac{U \sqrt{1 + \alpha^2}}{\alpha (\gamma_2 - \gamma_1)} \quad \forall x \in [0, \pi] \\
        &\Longleftrightarrow \quad U > \frac{\alpha}{2 \sqrt{1 + \alpha^2}} (\gamma_2 - \gamma_1).
    \end{align*}
    This gives the desired expression of $U_{\textrm{crit}}$.
    
    Let us now compute the minimal time in that case.     From the ODE \eqref{eq:theta_ode} satisfied by $\theta$ for the optimal trajectory, we know that
    \begin{align*}
        \mathcal{T}_U &= \int_0^\pi \frac{\diff \theta}{- \alpha (\gamma_2 - \gamma_1) \sin{\theta} \, \cos{\theta} + U \sqrt{1 + \alpha^2}} \\
            &= \int_0^\pi \frac{\diff \theta}{- \frac{1}{2} \alpha (\gamma_2 - \gamma_1) \sin{2 \theta} + U \sqrt{1 + \alpha^2}} \\
            &= \int_0^{2 \pi} \frac{\diff x}{- \alpha (\gamma_2 - \gamma_1) \sin{x} + 2 U \sqrt{1 + \alpha^2}} \\
            &= \frac{1}{2 \sqrt{1 + \alpha^2}} \int_{- \pi}^{\pi} \frac{\diff x}{- U_\textnormal{crit} \sin{x} + U}.
    \end{align*}
    With the change of variable $y = \tan \frac{x}{2}$, so that $\diff x = \frac{2 \diff y}{1 + y^2}$, $\sin x = \frac{2 y}{1 + y^2}$, we get
    \begin{align*}
        \mathcal{T}_U &= \frac{1}{2 \sqrt{1 + \alpha^2}} \int_{-\infty}^{+ \infty} \frac{2 \diff y}{- 2 U_\textnormal{crit} y + U (1 + y^2)} \\
            &= \frac{1}{\sqrt{1 + \alpha^2}} \int_{-\infty}^{+ \infty} \frac{\diff y}{U \Bigl( y - \frac{U_\textnormal{crit}}{U} \Bigr)^2 + \frac{U^2 - U_\textnormal{crit}^2}{U}} \\
            &= \frac{1}{U \sqrt{1 + \alpha^2}} \int_{-\infty}^{+ \infty} \frac{\diff y}{y^2 + \frac{U^2 - U_\textnormal{crit}^2}{U^2}} \\
            &= \frac{1}{U \sqrt{1 + \alpha^2}} \sqrt{\frac{U^2 - U_\textnormal{crit}^2}{U^2}} \int_{-\infty}^{+ \infty} \frac{\diff z}{\frac{U^2 - U_\textnormal{crit}^2}{U^2} z^2 + \frac{U^2 - U_\textnormal{crit}^2}{U^2}},
    \end{align*}
    with the change of variable $y = \sqrt{\frac{U^2 - U_\textnormal{crit}^2}{U^2}} z$. Therefore,
\[
        \mathcal{T}_U = \frac{1}{U \sqrt{1 + \alpha^2}} \sqrt{\frac{U^2}{U^2 - U_\textnormal{crit}^2}} \int_{-\infty}^{+ \infty} \frac{\diff z}{z^2 + 1} = \frac{\pi}{\sqrt{1 + \alpha^2} \sqrt{U^2 - U_\textnormal{crit}^2}}. \qedhere
\]
\end{proof}

\section{Proof of Theorem~\ref{th:alm_planar}: on almost planar trajectories}\label{proof:theo8}

Let us first state a result based on tedious computations, whose detail is left to the reader.

\begin{lemma} \label{lem:zeta_ode}
    Let $U > U_\textnormal{crit}$, $m$ be an optimal trajectory and $p$ its adjoint state. 
    
    Then $\zeta \coloneqq p \wedge m = \varphi \wedge m$ satisfies
    \begin{equation*}
        \dot{\zeta} = \alpha \Bigl( D (m \wedge \zeta) \wedge m - (m \wedge \zeta) \wedge Dm \Bigr) + Dm \wedge \zeta - D \zeta \wedge m
    \end{equation*}
    
    Similarly, denote $Z \coloneqq \zeta/\abs{\zeta}$. At every point where $\zeta$ does not vanish, one has 
    \begin{equation} \label{eq:Z_ode}
        \dot{Z} = \mathbb{P}_{Z^\perp} \biggl[ \alpha \Bigl( D (m \wedge Z) \wedge m - (m \wedge Z) \wedge Dm \Bigr) + Dm \wedge Z - D Z \wedge m \biggr],
    \end{equation}
    where $\mathbb{P}_{Z^\perp}: x\mapsto x - (Z \cdot x) Z$ is the projection onto the orthogonal space to $Z$.
\end{lemma}

The proof of the Theorem~\ref{th:alm_planar} relies on the following result, establishing the existence of a planar trajectory joining $e_1$ to $-e_1$, without any norm condition on the chosen control.

\begin{lemma}[Existence of 
planar trajectory]\label{lem:2151}
For any $\varepsilon >0$, there exists a trajectory of the form $m(t) = (m_1(t),m_2(t),0)$ defined on $[0,T_\varepsilon ]$ for some $T_\varepsilon >0$, joining the state $e_1$ to $-e_1$ and such that
\[
\mathcal F: = \left(\dot e\cdot (Dm\wedge m)+\frac{|\dot m|}{1+\alpha^2}\right)^2+\left(\dot e\cdot Dm+\frac{\alpha|\dot m|}{1+\alpha^2}\right)^2 \leq \frac{1}{4} (\gamma_2 - \gamma_1)^2 (1+\varepsilon).
\]
 for all $t\in [0,T_\varepsilon]$. Furthermore, $T_\varepsilon \lesssim 1/\varepsilon$.
\end{lemma}

 \begin{proof}[Proof of Lemma~\ref{lem:2151}]
Define $m_1(t) = \cos \theta(t)$ and $m_2(t) = \sin \theta(t)$, our goal is to define a suitable function $\theta$, such that $\theta(0)=0$ and $\theta(T_\varepsilon ) = \pi$.

Observe that $\dot e, Dm, m$ are coplanar so that $\dot e\cdot (Dm\wedge m) =0$. Also
\[ |\dot m| = |\dot \theta|, \quad \dot e \cdot Dm = (\gamma_2-\gamma_1) \sin(\theta) \cos(\theta) = \frac{\gamma_2-\gamma_1}{2} \sin(2\theta). \]
Hence
\begin{align*}
\mathcal F  &= \frac{1}{(1+\alpha^2)^2}|\dot \theta|^2 + \left( \frac{\gamma_2-\gamma_1}{2} \sin(2\theta) + \frac{\alpha |\dot \theta|}{1+\alpha^2} \right)^2 \\
& = \frac{1}{1+\alpha^2}|\dot \theta|^2 + \frac{\alpha (\gamma_2-\gamma_1)}{1+\alpha^2} \sin(2\theta) |\dot \theta| + \frac{(\gamma_2-\gamma_1)^2}{4} \sin^2(2\theta).
\end{align*}
This is a quadratic expression in $|\dot \theta|$. Let us solve $\mathcal F = \frac{1}{4} (\gamma_2 - \gamma_1)^2 (1+\varepsilon )$: this is a polynomial equation of degree 2, whose discriminant reads
\begin{align*}
\Delta & = \frac{\alpha^2 (\gamma_2-\gamma_1)^2}{(1+\alpha^2)^2} \sin^2(2\theta) - \frac{4}{(1+\alpha^2)^2} \frac{(\gamma_2-\gamma_1)^2}{4} ( \sin^2(2\theta) - 1 -\varepsilon ) \\
& =\frac{(\gamma_2-\gamma_1)^2}{(1+\alpha^2)^2} \left( 1+ \varepsilon + (\alpha^2-1) \sin^2(2\theta) \right).
\end{align*}
Observe that $\Delta >0$ for all $\theta$, so that we can choose
\[ \dot \theta = \frac{\gamma_2-\gamma_1}{2} \left( -\alpha \sin(2\theta) + \sqrt{1+ \varepsilon + (\alpha^2-1) \sin^2(2\theta) }\right) =: f_\varepsilon (\theta) \]
As 
\[ f_\varepsilon (\theta) \geq \frac{\gamma_2-\gamma_1}{2} \left( \sqrt{\varepsilon + \alpha^2 \sin^2(2 \theta)} - \alpha \sin(2\theta) \right) \geq \frac{\varepsilon }{\sqrt{\varepsilon + \alpha^2} + \alpha}>0, \] we infer that this ODE on $\theta$ admits a unique solution $\theta_\varepsilon $, strictly increasing such that $\dot \theta_\varepsilon \gtrsim \varepsilon $, and so, there exists a unique $T_\varepsilon \lesssim 1/\varepsilon $ such that $\theta_\varepsilon (T_\varepsilon ) = \pi$. This provides the desired trajectory.
\end{proof}

Denote $U_\textnormal{plan} = \frac{\gamma_2 - \gamma_1}{2}$. Lemma \ref{lem:2151} shows in particular that $U_{\text{crit}} \le U_{\text{plan}}$. We are now in position to complete the: 

\begin{proof}[Proof of Theorem~\ref{th:alm_planar}]
    Let $\zeta$ and $Z$ as in Lemma \ref{lem:zeta_ode}: $Z$ satisfies \eqref{eq:Z_ode}. Observe moreover that the equations on $\zeta$ and $Z$ remain unchanged if one replaces $D$ into $D - \lambda I_3$ for some $\lambda \in \mathbb{R}$. We can therefore assume that the spectral norm of $D$ is $\norm{D} = \frac{\gamma_3 - \gamma_1}{2}$ by taking $\lambda = \frac{\gamma_3 + \gamma_1}{2}$. 
    Then, since $\abs{Z} = \abs{m} = 1$ and $\norm{\mathbb{P}_{Z^\perp}} = 1$, we get
    \begin{equation*}
        \abs{\dot{Z}} \leq 2 (1 + \alpha) \norm{D} = (1 + \alpha) (\gamma_3 - \gamma_1).
    \end{equation*}
 On the other hand, according to Lemmas~\ref{lem:2057} and \ref{lem:2151}, we know that for all $U > U_\textnormal{plan}$, there holds $\mathcal{T}_U \leq \frac{C}{U - U_\textnormal{plan}}$ for some constant $C > 0$. Thus, one has
    \begin{equation}\label{eq:metz2114}
        \abs{Z(t) - Z (\mathcal{T}_U)} \leq (1 + \alpha) (\gamma_3 - \gamma_1) \mathcal{T}_U \leq (1 + \alpha) (\gamma_3 - \gamma_1) \frac{C}{U - U_\textnormal{plan}}.
    \end{equation}
    for all $t \in [0, \mathcal{T}_U]$.%
    We also know that $\abs{\zeta} = \abs{\varphi}$. Thus, by introducing $\psi=\varphi/|\varphi|$, one gets $Z = \psi \wedge m$ and $m \wedge Z = \psi$ since $\varphi \cdot m = 0$. 
   A straightforward computation yields that the pair $(m,\psi)$ satisfies
\begin{align}
\dot{\psi}& = \alpha (D\psi-(D\psi\cdot \psi)\psi)- U \sqrt{1+\alpha^2}m-Dm\wedge \psi+D(m\wedge \psi) \notag \\
& \qquad \qquad -\left(\psi\cdot D(m\wedge \psi)\right)\psi-\left((m\wedge \psi)\cdot Dm\right)m, \notag \\
\dot{m} &=  -\alpha \left(Dm-(Dm\cdot m)m\right)+m\wedge Dm+ U \sqrt{1+\alpha^2}\psi. \label{eq:ode_m_with_psi}
\end{align}

    From estimate \eqref{eq:metz2114}, we infer that, for $U$ large enough,
    \begin{equation*}
     \forall t \in [0,\mathcal T_U], \quad   \abs{\psi (t) - m \wedge Z (\mathcal{T}_U)} \leq \frac{C}{U}.
    \end{equation*}
    Putting this in Equation~\eqref{eq:ode_m_with_psi} for $m$, we get some constant $C > 0$ such that for all $U$ large enough and $t \in [0, \mathcal{T}_U]$,
    \begin{equation*}
        \abs{\dot{m} - U \sqrt{1 + \alpha^2} m \wedge Z (\mathcal{T}_U)} \leq C,
    \end{equation*}
    which leads to
    \begin{equation*}
        \abs{\dot{m} \cdot Z (\mathcal{T}_U)} \leq C.
    \end{equation*}
    Since $m (\mathcal{T}_U) \cdot Z (\mathcal{T}_U) = 0$ by definition and using once again that $\mathcal{T}_U \leq \frac{C}{U - U_\textnormal{plan}}$, we get for all $U$ large enough and $t \in [0, \mathcal{T}_U]$
    \begin{equation*}
        \abs{m \cdot Z (\mathcal{T}_U)} \leq \frac{C}{U}.
    \end{equation*}
    However, $p (\mathcal{T}_U) = \varphi (\mathcal{T}_U)$ (since $p (\mathcal{T}_U) \cdot m (\mathcal{T}_U) = 0$ from the orthogonality condition) and $m(\mathcal{T}_U) = - e_1$, and thus the orthogonal space of $V$ is exactly $\operatorname{span} (Z (\mathcal{T}_U))$, which means that $m (t) - \mathbb{P}_V m (t) = (m \cdot Z (\mathcal{T}_U)) Z (\mathcal{T}_U)$. The conclusion easily follows.
\end{proof}

%

\section{Conclusion and perspectives}
\label{sect:Conclusion and perspectives}
\subsection{Extension of our results}
It would be natural to extend our study in several directions. On the one hand, we would like to complete our study of a single ferromagnetic particle of ellipsoidal shape by studying other criteria, and typically a combination of time and cost $L^2$ of control. This problem could read:  

\begin{quote}
\textbf{Second version of the optimal control problem: case of $L^2$ constraints.}
Let $\lambda >0$ and let us assume that $m_0=e_1$. The problem reads
\begin{equation}\label{OCP:minTL2}\tag{$\mathscr{P}_\lambda$}
\mathcal{E}_U^\lambda = \inf_{(T,u)\in \mathcal{O}_{U }}T+\frac{\lambda}{2} \int_0^T \abs{u(t)}^2\, dt,
\end{equation}
where $m_u$ denotes the solution to \eqref{LL:ODE} associated to the control function $u(\cdot)$,
\end{quote} 

or alternatively, if one aims at dropping the effect of the $L^\infty$ constraint on the control, 

\begin{quote}
\textbf{Modified second version of the optimal control problem: case of $L^2$ constraints.}
Let $\lambda >0$ and let us assume that $m_0=e_1$. The problem reads
\begin{equation}\label{OCP:minTL2_bis}\tag{$\mathscr{P}_\lambda$}
\mathcal{E}_U^\lambda = \inf_{(T,u)\in \bigcup_{U\geq 0}\mathcal{O}_{U }}T+\frac{\lambda}{2} \int_0^T \abs{u(t)}^2\, dt,
\end{equation}
where $m_u$ denotes the solution to \eqref{LL:ODE} associated to the control function $u(\cdot)$.
\end{quote} 

Finally, we also plan to study similar issues for more realistic physical systems, for example a network of ellipsoidal particles, possibly rectilinear, as in the model introduced in \cite{agarwal2011control}.

\subsection{Numerical illustrations of our results}
We provide hereafter several numerical illustrations of our results. More precisely, we want to determine numerically the existence or not of an admissible trajectory connecting $e_1$ to $-e_1$, in accordance with what we have found theoretically. Let us first notice that a trajectory $m$ can easily be computed numerically by solving the ODE \eqref{LL:ODE} with the expression \eqref{expr:uopt1809} for the control $u$ where the variable $\varphi$ is given by \eqref{eq1155}.

To initialize both ODE \eqref{LL:ODE} and \eqref{eq1155}, $m(0)=e_1$ is given, but $
\varphi(0)$ is unknown. On the one hand, we overcome this difficulty by noticing that $\varphi(0).e_1=0$, which allows us to have only two unknowns: $\varphi_2(0)$ and $\varphi_3(0)$ to be determined.  On the other hand, working with the normalized  variable $\psi=\varphi/|\varphi|$ enables us to reduce the unknowns to only one angle variable $\vartheta\in [0, 2\pi]$ such that $(\psi_2(0), \psi_3(0))=(\cos(\vartheta), \sin(\vartheta))$. ODE \eqref{LL:ODE} and \eqref{eq1155} are thus replaced by the system \eqref{eq:ode_m_with_psi}.

Numerically, implement a shooting method to determine $\vartheta \in [0, 2\pi]$: namely, for each $\vartheta$, we solve the system \eqref{eq:ode_m_with_psi} on a very large time horizon by a fourth-order Runge-Kutta method  and determine if the trajectory $m$ reaches $-e_1$ on a certain time $T$.

We list below the numerical results, all obtained with $\alpha=0.6$. The initial position $e_1$ is represented with a red circle on the sphere and the goal $-e_1$ with a green star. Parameters $\gamma_i$, $\vartheta$ and control $U$ are specified in the caption of each figure. For each one, we have represented the trajectory $m$ on the sphere as well as the coordinates of $m$ and of the control $u$ as functions of time.

\begin{figure}[htbp]
    \centering
    \subfigure[Admissible trajectory for $\vartheta=0.8976$, large control $U=10$,  \label{fig:non_sym_big_U_sphere}]{\includegraphics[width=0.35\textwidth]{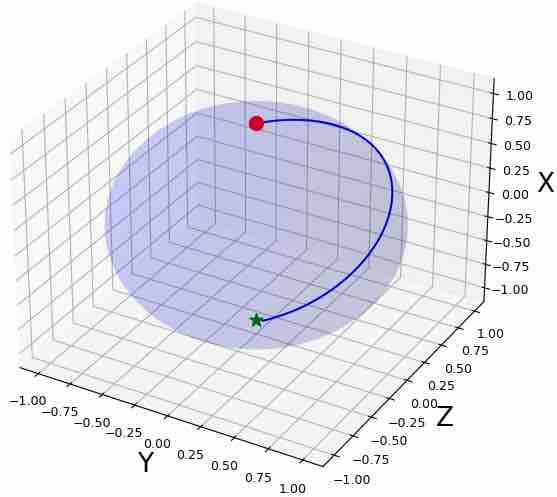}}
    \subfigure[The components of $m$ and $u$ for $\vartheta=0.8976$, large control $U=10$\label{fig:non_sym_big_U_compo}]{\includegraphics[width=0.6\textwidth]{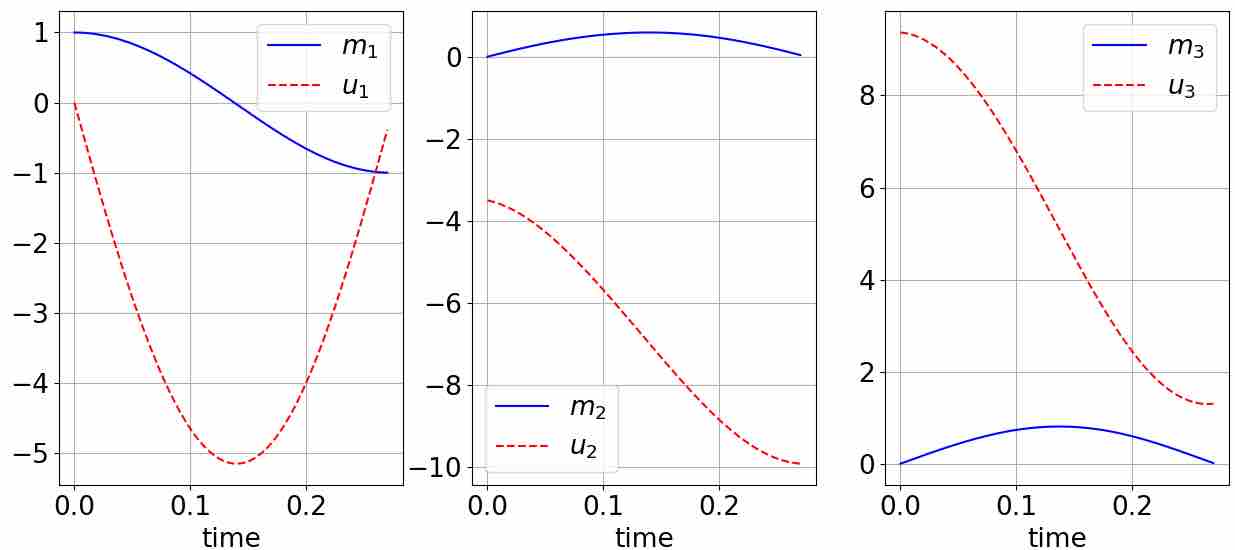}}\\
    \subfigure[Admissible trajectory for $\vartheta=2.2440$, medium control $U=3$\label{fig:non_sym_medium_U_sphere}]{\includegraphics[width=0.35\textwidth]{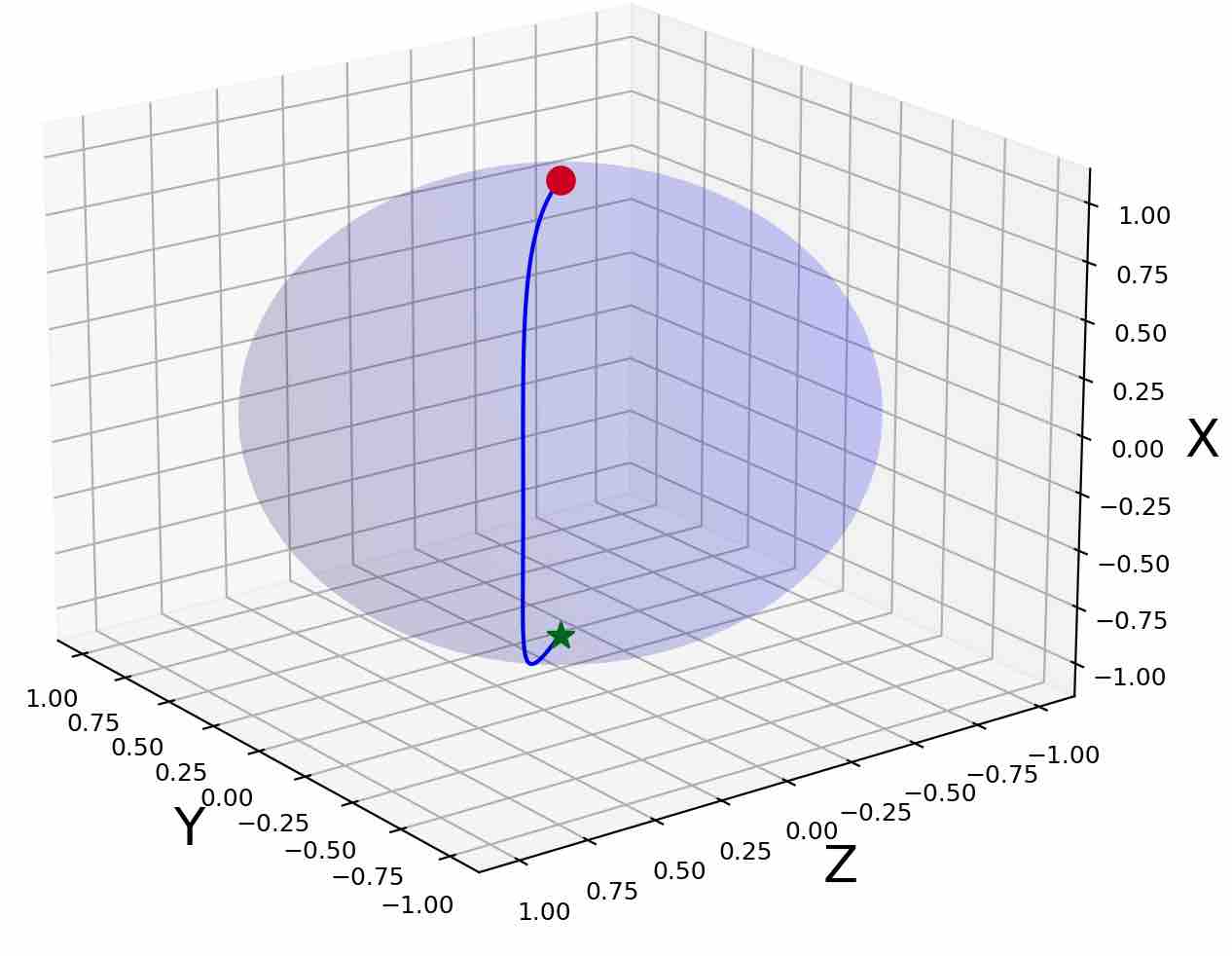}}
    \subfigure[The components of $m$ and $u$ for $\vartheta=2.2440$, medium control $U=3$\label{fig:non_sym_medium_U_compo}]{\includegraphics[width=0.6\textwidth]{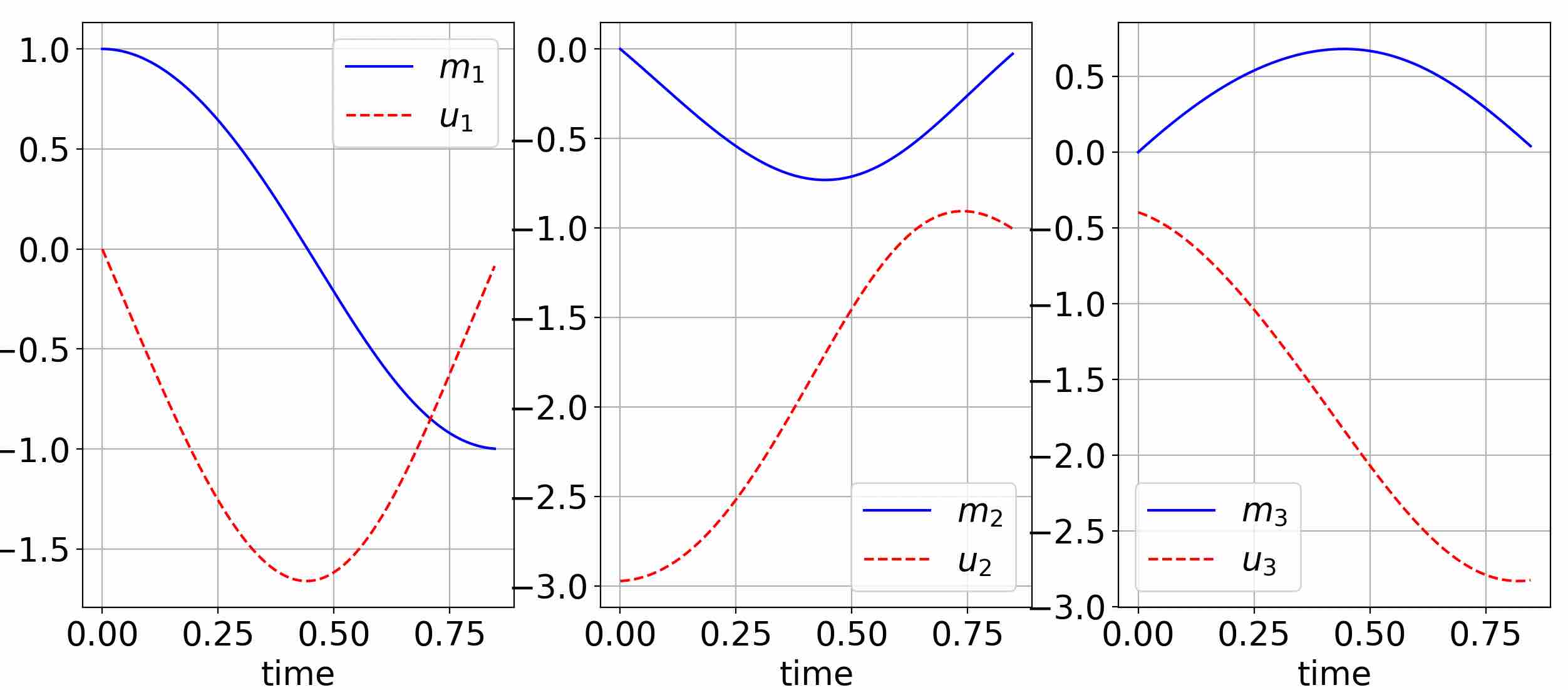}}\\
    \subfigure[Generic trajectory for $\vartheta=0.8976$, small control $U=0.1$ \label{fig:non_sym_small_U_sphere}]{\includegraphics[width=0.35\textwidth]{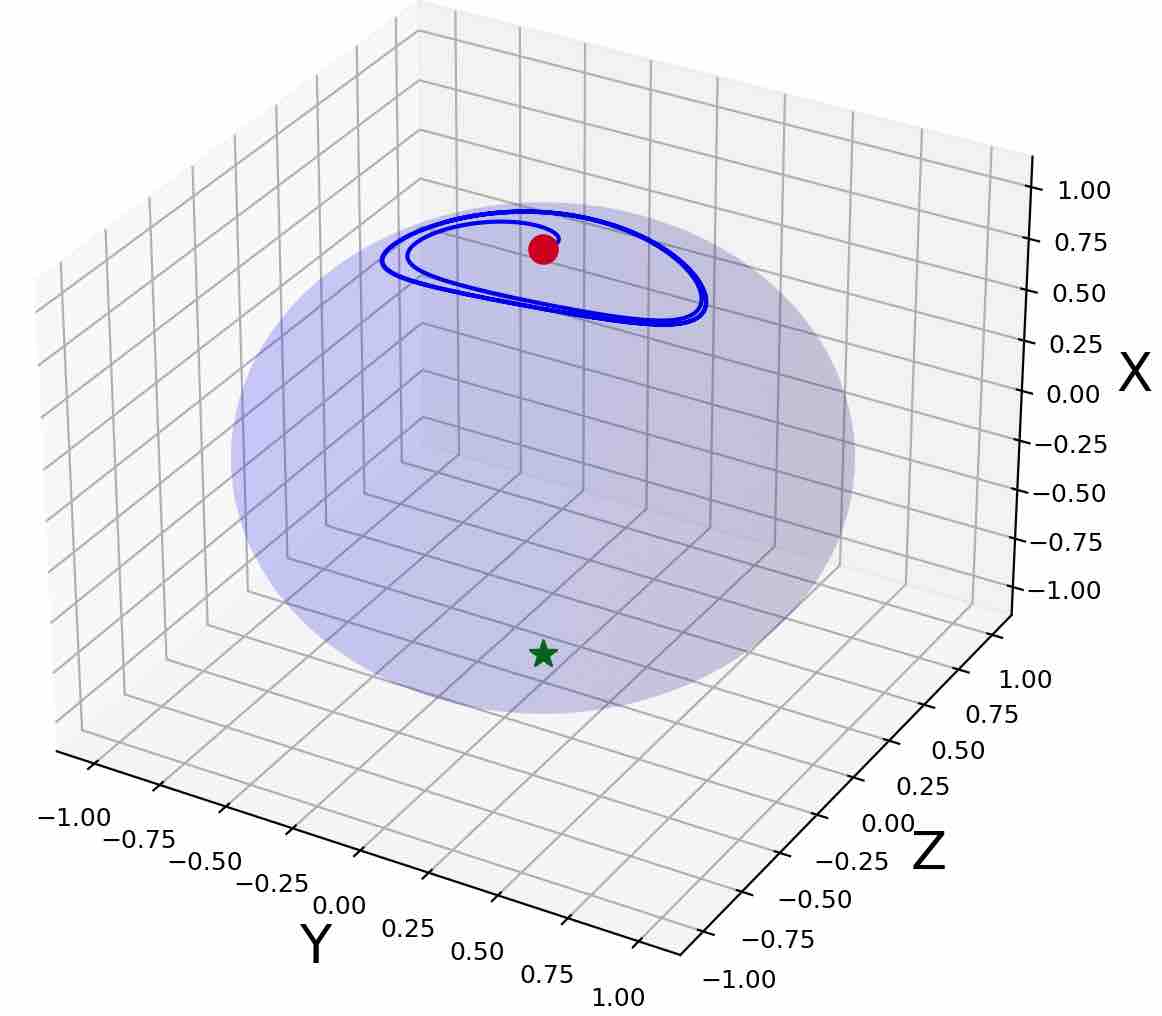}}
    \subfigure[The components of $m$ and $u$ for $\vartheta=0.8976$, small control $U=0.1$ \label{fig:non_sym_small_U_compo}]{\includegraphics[width=0.6\textwidth]{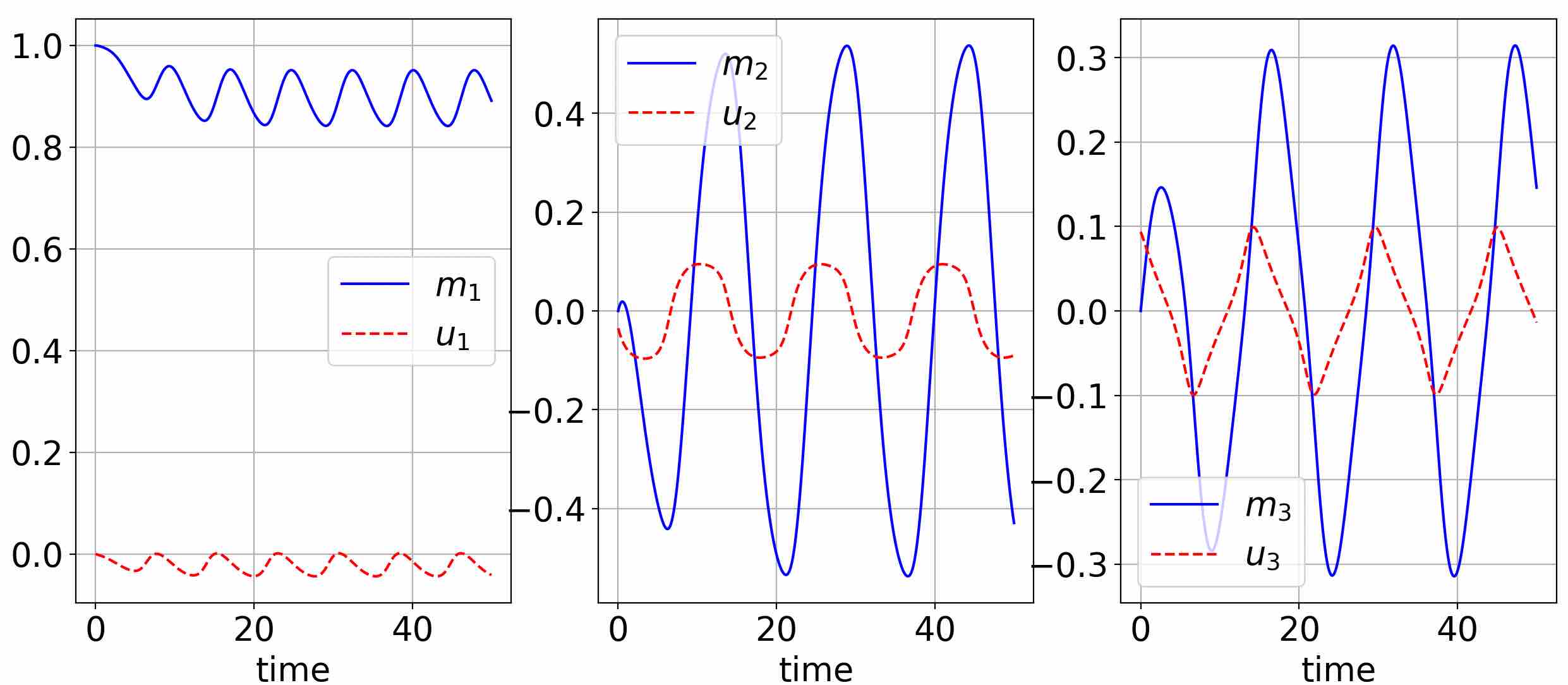}}
    \caption{Non-symmetric test case with $(\gamma_1; \gamma_2; \gamma_3) = (0.2; 0.5; 1)$, top: with a large control $U=10$, middle: with a medium control $U=3$ and bottom: with a small control $U=0.1$}
    \label{fig:non_sym_differents_U}
\end{figure}



First of all, in the non-symmetric case, a threshold on the control appears. If the control is sufficiently large, there is (at least) an initialization of $\psi$ (i.e at least one angle $\vartheta$) which allows to have an admissible trajectory represented in Subfigures \ref{fig:non_sym_big_U_sphere}-\ref{fig:non_sym_big_U_compo}. 
On the contrary, if the control is not large enough, no initialization of $\psi$ will give an admissible trajectory. We have represented for instance one of them in Subfigure \ref{fig:non_sym_small_U_sphere}-\ref{fig:non_sym_small_U_compo} with a particular $\vartheta$ but be aware that they all have the same behavior whatever the initialization of $\psi$: the trajectory remains in the northern half-sphere without enough control.
Figure \ref{fig:non_sym_differents_U} is thus a perfect illustration of Theorem \ref{th:ex_U_crit}. Note that it also helps to illustrate Theorem \ref{th:alm_planar} since the larger $U$ is, the closer the trajectory is to a planar trajectory, as we can see by comparing Subfigures \ref{fig:non_sym_medium_U_sphere}-\ref{fig:non_sym_medium_U_compo} with a medium control and Subfigures \ref{fig:non_sym_big_U_sphere}-\ref{fig:non_sym_big_U_compo} with a larger control.

\begin{figure}[htbp]
    \centering
    \subfigure[Generic trajectory]{\includegraphics[width=0.35\textwidth]{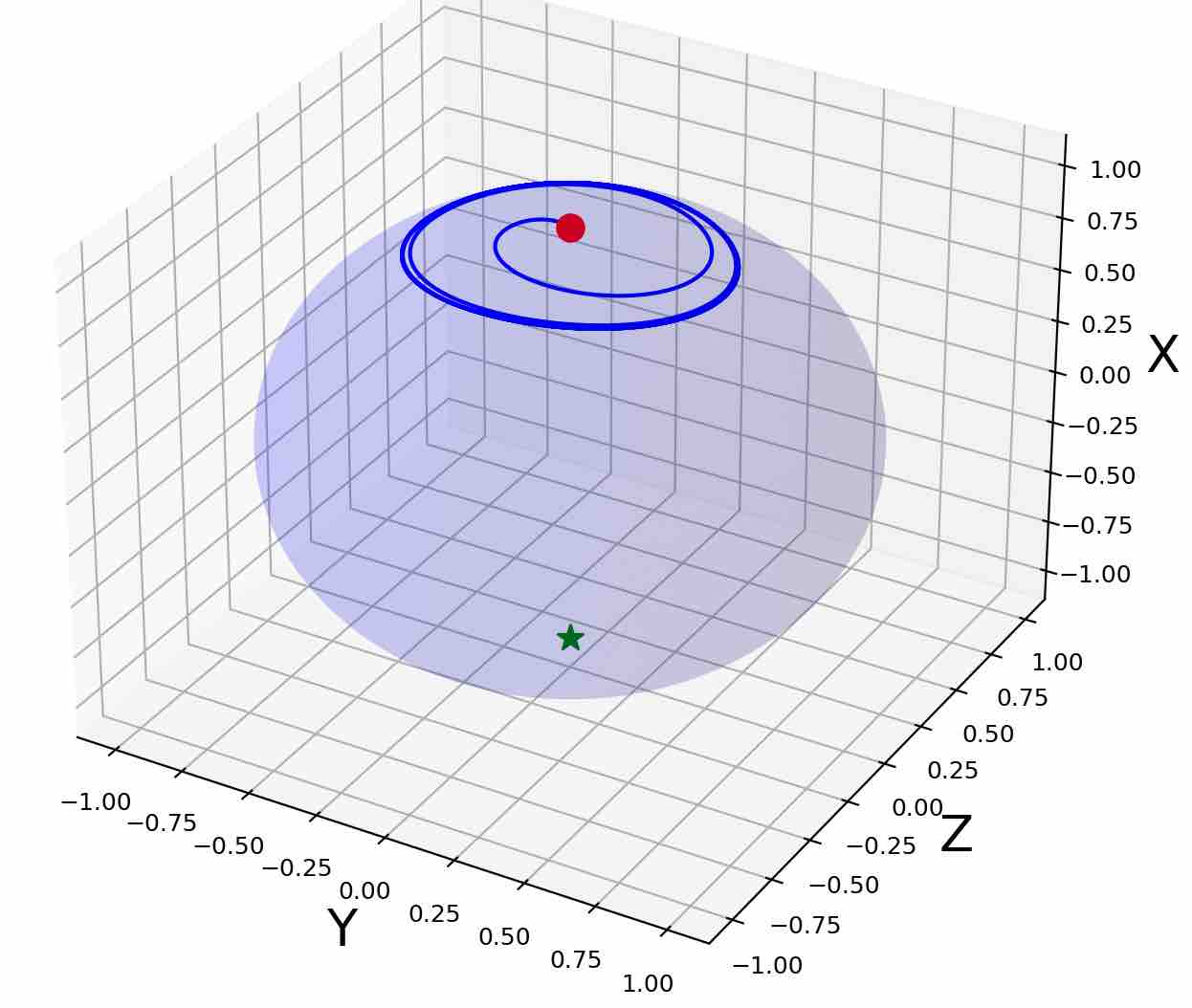}}
    \subfigure[The components of $m$ and $u$]{\includegraphics[width=0.6\textwidth]{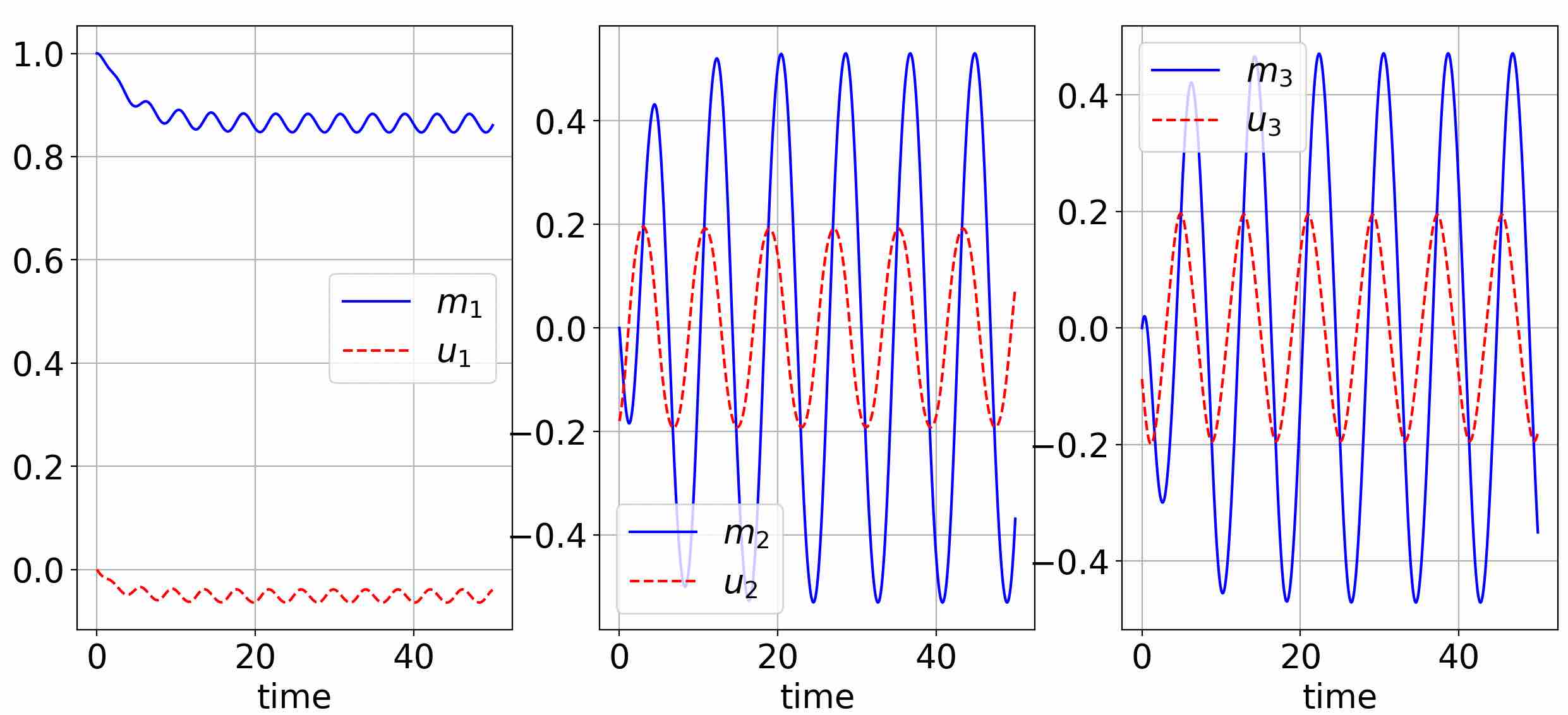}}
    \caption{Non-symmetric test case with $(\gamma_1; \gamma_2; \gamma_3) = (0.0; 0.8; 1)$, $\vartheta=2.5646$ and a small control $U=0.2$}
    \label{fig:petit_U}
\end{figure}
Figure \ref{fig:petit_U} illustrates once again the case of control too weak to reach $-e_1$, for other $\gamma_i$ parameters.

\bigskip

The symmetric case $\gamma_1 = \gamma_2$ is shown in Figure \ref{fig:sym_12}. Even for small controls ($U=0.7$ numerically), there is (at least) one initialization of $\psi$ leading to an admissible trajectory reaching $-e_1$ in finite time. This illustrates well Theorem \ref{theo:gamma1=gamma2}: $U_{\text{crit}} = 0$ in this symmetric case. When $\gamma_2<\gamma_3$ (Subfigures \ref{fig:sym_12_notspherical_sphere}-\ref{fig:sym_12_notspherical_compo}), the admissible trajectories are non planar whereas it is, in the case of a spherical symmetry (Subfigures \ref{fig:sym_12_spherical_sphere}-\ref{fig:sym_12_spherical_compo}) without changing anything other than the symmetry of the test case. This is again in accordance with the second statement of Theorem \ref{theo:gamma1=gamma2}.

\bigskip

\begin{figure}[htbp]
    \centering
    \subfigure[Admissible trajectory with $\gamma_3=1.0$\label{fig:sym_12_notspherical_sphere}]{\includegraphics[width=0.35\textwidth]{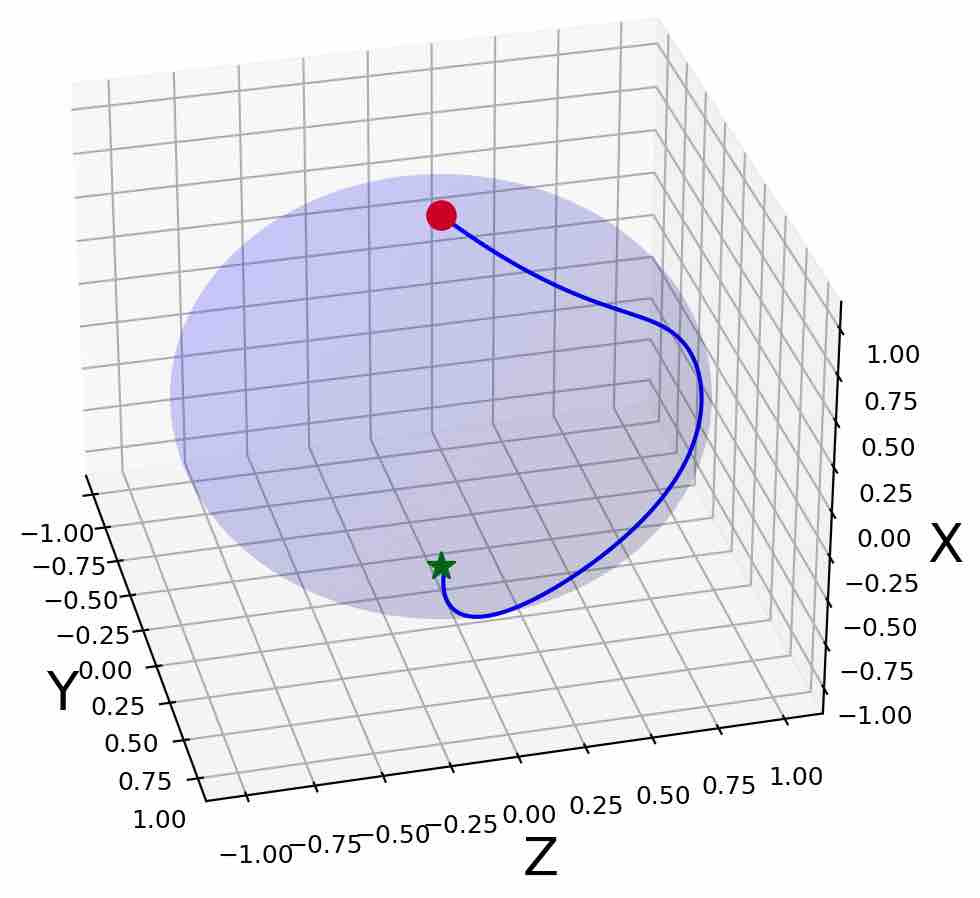}}
    \subfigure[The components of $m$ and $u$ with $\gamma_3=1.0$\label{fig:sym_12_notspherical_compo}]{\includegraphics[width=0.6\textwidth]{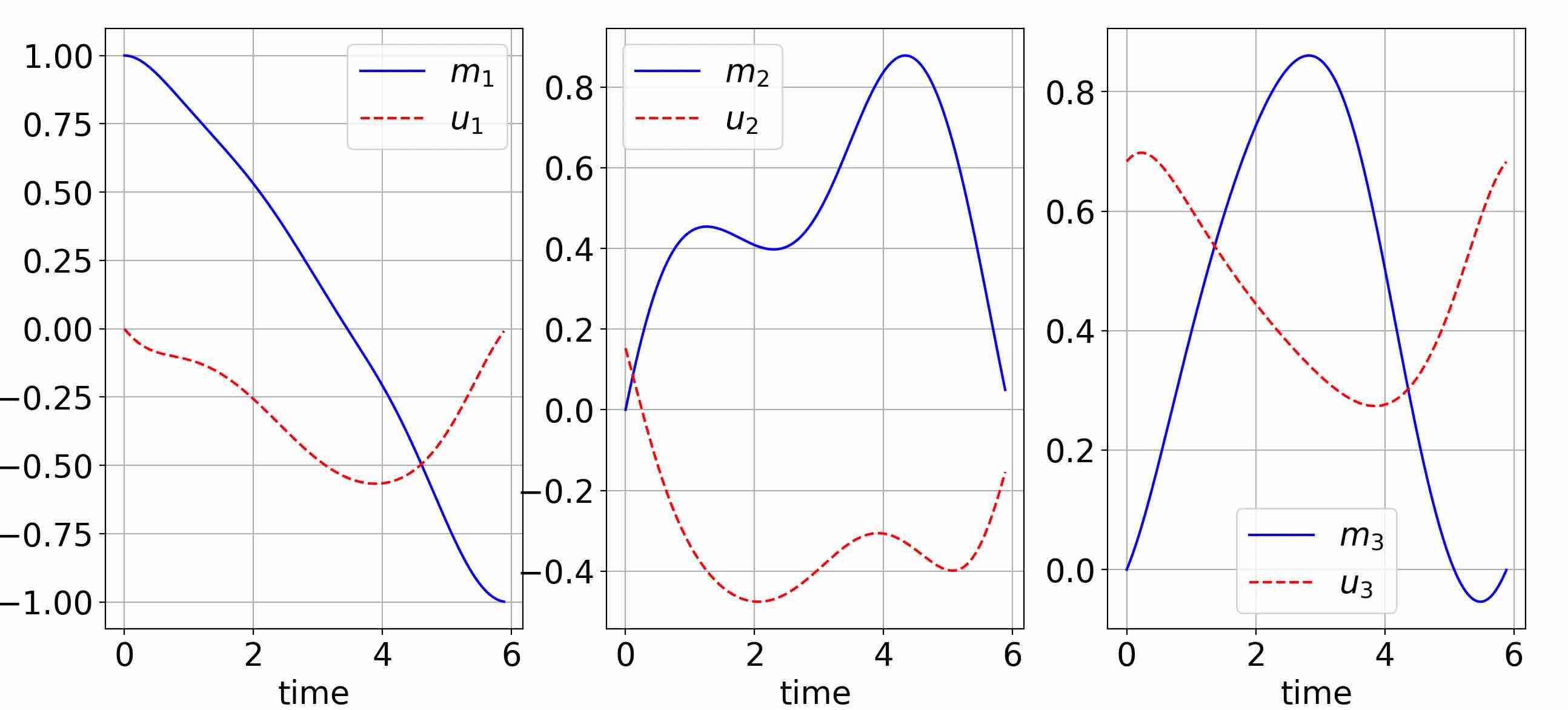}}
    \subfigure[Admissible trajectory with $\gamma_3=0.2$ (spherical case)\label{fig:sym_12_spherical_sphere}]{\includegraphics[width=0.35\textwidth]{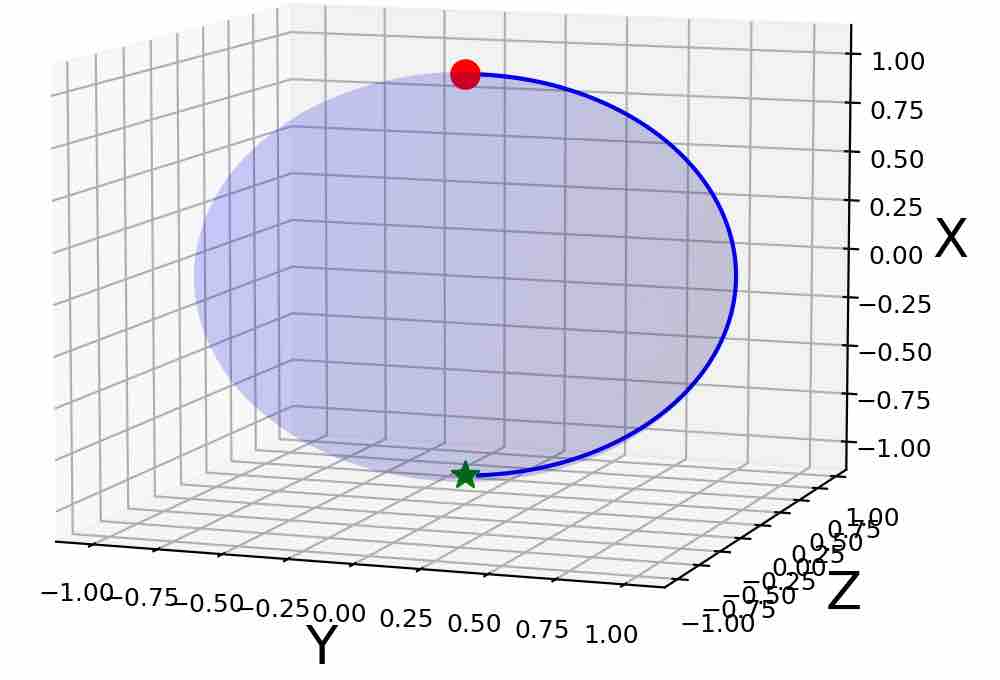}}
    \subfigure[The components of $m$ and $u$ with $\gamma_3=0.2$ (spherical case)\label{fig:sym_12_spherical_compo}]{\includegraphics[width=0.6\textwidth]{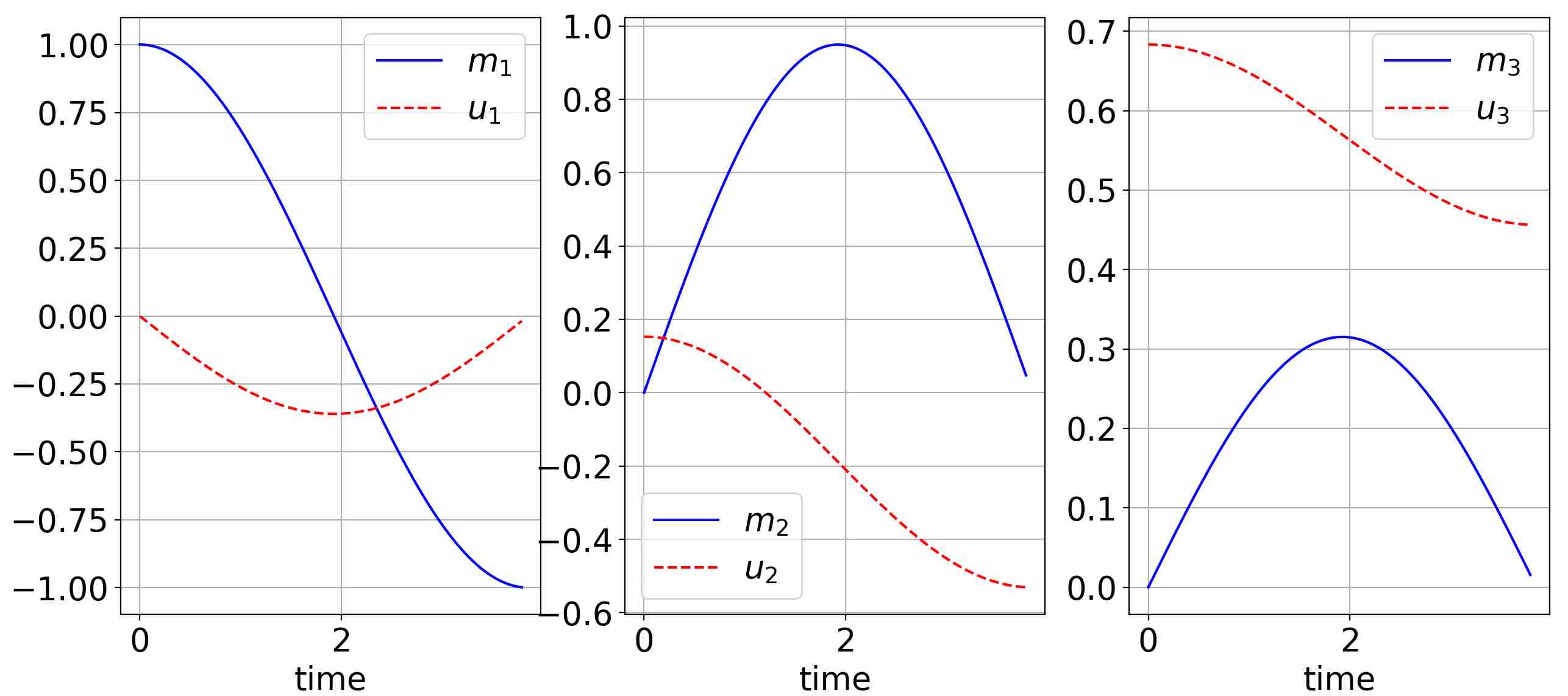}}
    \caption{Symmetric test case $\gamma_1=\gamma_2$ with $(\gamma_1; \gamma_2) = (0.2; 0.2)$, $\vartheta=0.3206$ and a small control $U=0.7$, top: $\gamma_3=1.0$ and bottom: $\gamma_3=0.2$ (spherical case)}
    \label{fig:sym_12}
\end{figure}


\begin{figure}[htbp]
    \centering
    \subfigure[Admissible trajectory]{\includegraphics[width=0.35\textwidth]{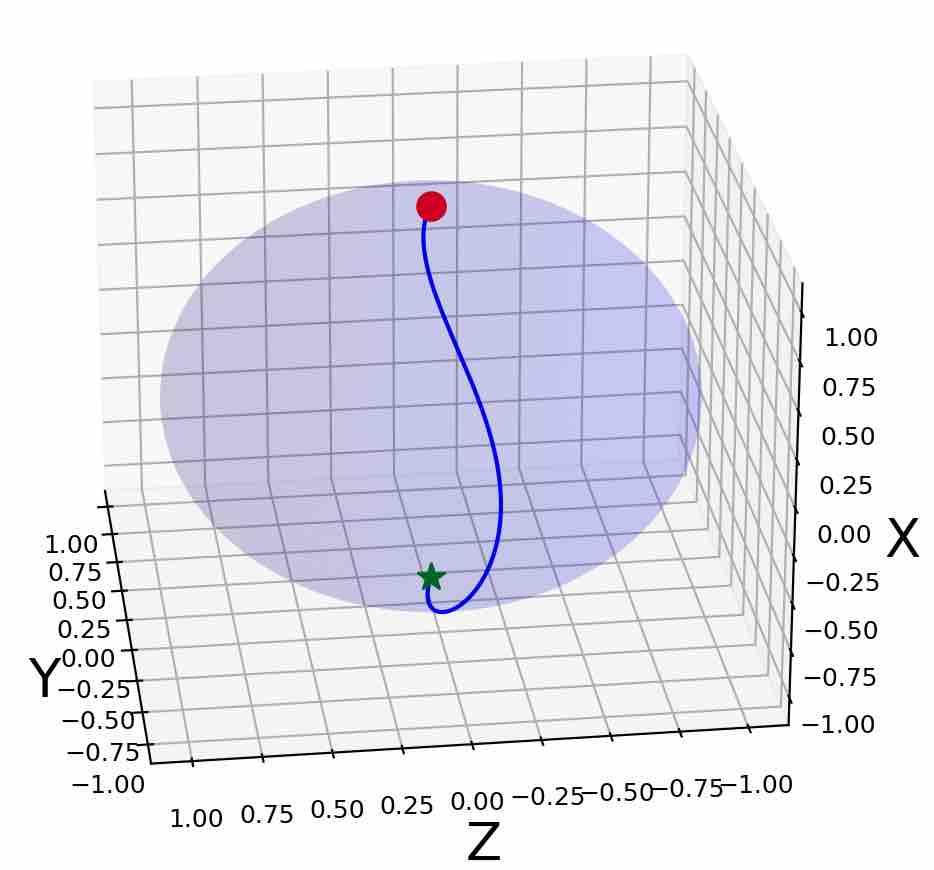}}
    \subfigure[The components of $m$ and $u$\label{fig:sym_23_small_U_compo}]{\includegraphics[width=0.6\textwidth]{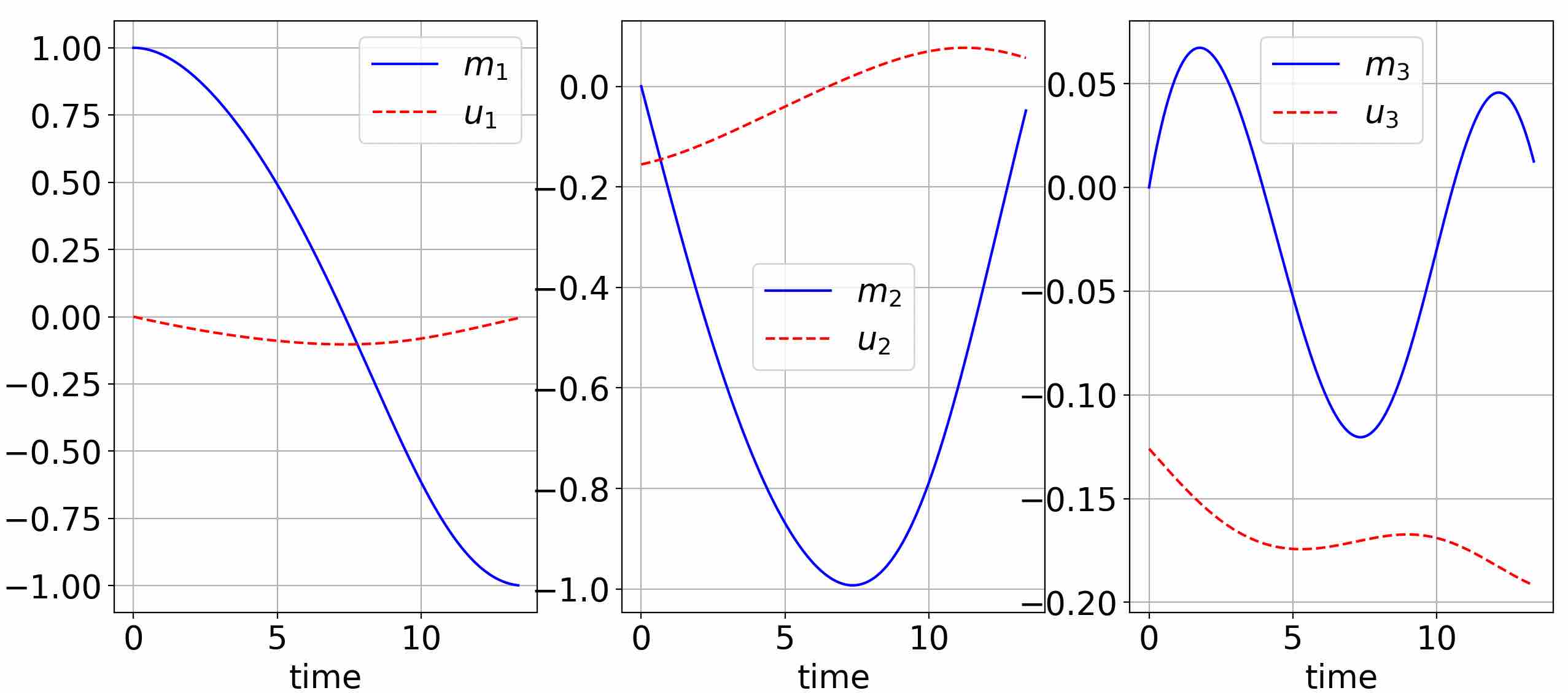}}
    \caption{Symmetric $\gamma_2=\gamma_3$ test case with $(\gamma_1; \gamma_2; \gamma_3) = (0.1; 0.2; 0.2)$, $\vartheta=2.7925$ and a small control $U=0.2$}
    \label{fig:sym_23_small_U}
\end{figure}

For the symmetric case $\gamma_2=\gamma_3$, we see numerically in Figure \ref{fig:sym_23_small_U} that for small values of $U$, an admissible trajectory exists. With the parameters of Figure \ref{fig:sym_23_small_U}, Theorem \ref{theo:gamma2=gamma3} gives the following value for $U_{\text{crit}} = \frac{\alpha}{2\sqrt{1+\alpha^2}}(\gamma_2-\gamma_1)\simeq 0.026$, which effectively allows to have admissible trajectories for very small values of $U$. Note also in Subfigure \ref{fig:sym_23_small_U_compo} that all admissible trajectories reach the target $-e_1$ in a time greater than 14. With the values chosen for Figure \ref{fig:sym_23_small_U}, $\frac{\pi}{\sqrt{1+\alpha^2}\sqrt{U^2-U_{\text{crit}}^2}}\simeq 13.58$ corresponds to the minimum time determined in Theorem \ref{theo:gamma2=gamma3}. Here again, we notice the non-planar character of the trajectory.

\subsection{Conclusion and perspectives}
The obtained results provide a complete characterization of the question of the magnetic moment reversal in minimal time in a simple configuration. Indeed, we have considered here only one ellipsoidal particle. In order to approach more realistic configurations, we wish to analyze a model in which several ferromagnetic particles of ellipsoidal shape are combined to form a network. We refer for example to \cite{MR2833256} for a possible model. After having characterized the set of stationary configurations, we will then ask ourselves the question of controllability in minimal time, in order to go from one stationary state to another.

\section*{Acknowledgements}
The authors were partially supported by the ANR Project MOSICOF. The last author were partially supported by the ANR Project TRECOS.

\appendix
\section{Computation of the demagnetizing field in a ferromagnetic ellipsoid sample}\label{append:demag_ellips}

It is shown in \cite{osborn1945demagnetizing,di2016newtonian} that the demagnetizing tensor $D$ reads $D=\operatorname{diag}([\gamma_1,\gamma_2,\gamma_3])$, where the $\gamma_i$'s are given by
\[
\gamma_i=\frac{a_1a_2a_3}{2} \int_0^{+\infty}\frac{dt}{\sqrt{(a_1+t^2)(a_2^2+t^2)(a_3^2+t^2)}(a_i^2+t^2)}.
\]
Such an expression can be rewritten in terms of the elliptic integral of the second kind $E$, defined by 
\[
E(x,p)=\int_0^x (1-p\sin^2\theta)^{1/2}\, d\theta, \quad x\in \R, \ p\in (0,1).
\]
If $a_1\geq a_2\geq a_3$, then, one has $0\leq \gamma_1\leq \gamma_2\leq \gamma_3\leq 1$ and these coefficients read
\begin{align*} 
\gamma_1 &= 1-\gamma_2-\gamma_3\\
\gamma_2 &= -\frac{a_3}{a_2^2-a_3^2}\left(a_3-\frac{a_1a_2}{(a_1^2-a_2^2)^{1/2}}E\left(\frac{a_2}{a_1},\frac{a_1^2-a_3^2}{a_1^2-a_2^2}\right)\right)\\
\gamma_3 &= \frac{a_2}{a_2^2-a_3^2}\left(a_2-\frac{a_1a_3}{(a_2^2-a_3^2)^{1/2}}E\left(\frac{a_3}{a_1},\frac{a_1^2-a_2^2}{a_1^2-a_3^2}\right)\right).
\end{align*}
In the case where $a_1\geq a_2=a_3$ (\emph{prolate spheroid}), these formula simplify into
 \[
 \gamma_1= -\frac{a_3^2}{(a_1^2-a_3^2)^{3/2}}\left((a_1^2-a_3^2)^{1/2}+a_1\operatorname{argcoth}\left(\frac{a_1}{(a_1^2-a_3^2)^{1/2}}\right) \right), \quad \gamma_2=\gamma_3=\frac{1-\gamma_1}{2}.
\]
In the case where $a_1= a_2\geq a_3$ (\emph{oblate spheroid}), these formula simplify into
\[
 \gamma_3= -\frac{a_1^2}{(a_1^2-a_3^2)^{3/2}}\left((a_1^2-a_3^2)^{1/2}+a_3\arctan\left(\frac{a_3}{(a_1^2-a_3^2)^{1/2}}\right) -\frac{\pi}{2}a_3\right), \quad \gamma_1=\gamma_2=\frac{1-\gamma_3}{2}.
\]

\section{Stability of steady-states for Eq.~\texorpdfstring{\eqref{LL:ODE}}{(3)}}\label{append:e1AS}
Let us first notice that, if $\bar m$ is a steady-state of Equation~\eqref{LL:ODE} and if no control is applied on this system ($h_{\rm ext}=0$), then by orthogonality of the terms in the right-hand side, it satisfies 
\[
h_0(\bar m)=(h_0(\bar m)\cdot \bar m)\bar m\quad \text{and}\quad \bar m\wedge h_0(\bar m)=0. 
\]
Therefore, $\bar m$ is an eigenfunction of $D\bar m$ and we infer that $\bar m=\pm e_j$, $j=1,2,3$ whenever $\gamma_1<\gamma_2\leq \gamma_3$.
\begin{proposition}[Asymptotic stability] \label{prop:asympt_stab}
Let $\gamma_1, \gamma_2, \gamma_3$ be sorted in ascending order $\gamma_1<\gamma_2\leq\gamma_3$. Then, in the absence of any control $h_{\rm ext}$, the steady-state $\pm e_1$ is an asymptotically stable equilibrium state for Equation~\eqref{LL:ODE}. Nevertheless, the steady-states $\pm e_2$ and $\pm e_3$ are linearly unstable steady-states for Equation~\eqref{LL:ODE}.
\end{proposition}
\begin{proof}
Let $h=(h_1, h_2, h_3)^T\in\R^3$ be a small perturbation such that $e_1+h$ is still an admissible magnetization, i.e. on the unit sphere $\mathbb{S}^2\subset\R^3$. We obtain:
\begin{equation*}
\|e_1+h\|^2=1\Leftrightarrow (1+h_1)^2+h_2^2+h_3^2=1
\Leftrightarrow h_1=-\frac{1}{2}\left(h_1^2+h_2^2+h_3^3\right)=\mathcal{O}(\|h\|^2).
\end{equation*}
The unknown $h_1$ is therefore of second order and does not occur in a linearized system of the first order.

By linearizing Equation~\eqref{LL:ODE} around the equilibrium state $e_1$, one has, without any control $u$: 
\begin{equation}
\left\{
\begin{aligned}
\dot{h_2} & =\alpha(\gamma_1-\gamma_2)h_2+(\gamma_1-\gamma_3)h_3+\operatorname{O}(\|h\|^2)\\
\dot{h_3} & =\alpha(\gamma_1-\gamma_3)h_3+(\gamma_2-\gamma_1)h_2+\operatorname{O}(\|h\|^2)\\
\end{aligned}
\right.
\end{equation}

The Jacobian matrix of the linearized system around $e_1$ is therefore :
\[ J=\begin{pmatrix}\alpha(\gamma_1-\gamma_2)&\gamma_1-\gamma_3\\
\gamma_2-\gamma_1&\alpha(\gamma_1-\gamma_3)\end{pmatrix}.\]

Since $\gamma_1<\gamma_2\leq\gamma_3$, one has 
\[
\det (J)=(\alpha^2+1)(\gamma_1-\gamma_2)(\gamma_1-\gamma_3)>0\quad \text{and}\quad 
\operatorname{Tr}(J)=\alpha\left[(\gamma_1-\gamma_2)+(\gamma_1-\gamma_3)\right]<0.
\]
We infer that the two eigenvalues of the Jacobian matrix are of negative real parts. The steady state $e_1$ is therefore linearly stable and is an hyperbolic point (no eigenvalue with zero real part).

Hartman Grobman's theorem \cite{perko2013differential} allows to conclude about the asymptotic stability of $e_1$ for the non-linear Equation~\eqref{LL:ODE} without any control $u$. As for $- e_1$, similar computations give the conclusion. Regarding now the stability of $\pm e_k$, $k=2,3$, notice that a similar computation drives to the following expression of the Jacobian determinant: $\det J=(\alpha^2+1)(\gamma_2-\gamma_1)(\gamma_2-\gamma_3)<0$. The expected conclusion follows. 
\end{proof}

\begin{remark}
If $\gamma_1\leq\gamma_2\leq\gamma_3$, an eigenvalue of the Jacobian matrix may have a zero real part. In which case one can conclude that $e_1$ is linearly (non-asymptotically) stable, but Hartman Grobman's theorem no longer applies to return to the non-linear Equation~\eqref{LL:ODE}. 
\end{remark}

\section{Complement in the case \texorpdfstring{$\gamma_1<\gamma_2$}{gamma1<gamma2}: explicit computations of the constants in the case }\label{append:comp}

Let us use the notations introduced in Remark~\ref{rk:theoremUcrit}. We compute
\begin{equation*}
    A^* A = 
        \begin{bmatrix}
            (1 + \alpha^2) \delta \gamma_-^2 & - 2 \alpha \delta \gamma_- \, \delta \gamma_+ \\
            - 2 \alpha \delta \gamma_- \, \delta \gamma_+ & (1 + \alpha^2) \delta \gamma_+^2
        \end{bmatrix},
\end{equation*}
\begin{equation*}
    \operatorname{Tr}(A^* A) = (1 + \alpha^2) (\delta \gamma_-^2 + \delta \gamma_+^2) > 0, \qquad
    \det(A^* A) = (1 - \alpha^2)^2 \delta \gamma_-^2 \delta \gamma_+^2 \geq 0, 
\end{equation*}
and the discriminant of its characteristic polynomial is
\begin{align*}
    \operatorname{Tr}(A^* A)^2 - 4 \det(A^* A) &= (1 + \alpha^2)^2 (\delta \gamma_-^2 + \delta \gamma_+^2)^2 - 4 (1 - \alpha^2)^2 \delta \gamma_-^2 \delta \gamma_+^2 \\
        &= (1 + \alpha^2)^2 (\delta \gamma_-^2 - \delta \gamma_+^2)^2 + 16 \alpha^2 \delta \gamma_-^2 \delta \gamma_+^2 \\
        &= (1 + \alpha^2)^2 (\delta \gamma_- - \delta \gamma_+)^2 (\delta \gamma_- + \delta \gamma_+)^2 + 16 \alpha^2 \delta \gamma_-^2 \delta \gamma_+^2 > 0.
\end{align*}
and its largest eigenvalue is therefore
\begin{equation*}
    \norm{A}_2^2 = \frac{(1 + \alpha^2) (\delta \gamma_-^2 + \delta \gamma_+^2) + \sqrt{(1 + \alpha^2)^2 (\delta \gamma_- - \delta \gamma_+)^2 (\delta \gamma_- + \delta \gamma_+)^2 + 16 \alpha^2 \delta \gamma_-^2 \delta \gamma_+^2}}{2}.
\end{equation*}
On the other hand, when $\Delta = \alpha^2 (\delta \gamma_+ - \delta \gamma_-)^2 - \delta \gamma_- \delta \gamma_+ \geq 0$, we have
\begin{equation*}
    \lambda_+ = \frac{- \alpha (\delta \gamma_+ + \delta \gamma_-) + \sqrt{\alpha^2 (\delta \gamma_+ - \delta \gamma_-)^2 - \delta \gamma_- \delta \gamma_+}}{2},
\end{equation*}
and therefore, with $\Gamma = \delta \gamma_+^{-1} \delta \gamma_-$,
\begin{align*}
    \frac{\abs{\lambda_+}^2}{\norm{A}_2 (1 + \abs{\alpha}) \delta \gamma_+} &= \frac{1}{\sqrt{2} (1 + \abs{\alpha})} \frac{\Bigl( \alpha (1 + \Gamma) - \sqrt{\alpha^2 (1 - \Gamma)^2 - \Gamma} \Bigr)^2}{\Bigl( (1 + \alpha^2) ( 1 + \Gamma^2 ) + \sqrt{(1+\alpha^2) (\Gamma - 1)^2 (1 + \Gamma)^2 + 16 \alpha^2 \Gamma^2} \Bigr)^\frac{1}{2}}.
\end{align*}
Similarly, when $\Delta < 0$, we obtain
\begin{align*}
    \frac{\abs{\Tr A}^2}{\norm{A}_2 (1 + \abs{\alpha}) \delta \gamma_+} &= \frac{1}{\sqrt{2} (1 + \abs{\alpha})} \frac{\alpha (1 + \Gamma)}{\Bigl( (1 + \alpha^2) ( 1 + \Gamma^2 ) + \sqrt{(1+\alpha^2) (\Gamma - 1)^2 (1 + \Gamma)^2 + 16 \alpha^2 \Gamma^2} \Bigr)^\frac{1}{2}}.
\end{align*}
Remark also that $\Delta = \delta \gamma_+^2 \Bigl( \alpha^2 (1 - \Gamma)^2 - \Gamma \Bigr)$. Therefore, if we define
\begin{equation*}
    \tilde x_0 \coloneqq
        \begin{cases}
            \frac{\abs{\lambda_+}^2}{\norm{A}_2 (1 + \abs{\alpha}) \delta \gamma_+}, \qquad &\textnormal{if } \Delta \geq 0, \\
            \frac{\Tr (A)^2}{4 \norm{A}_2 (1 + \abs{\alpha}) \delta \gamma_+}, \qquad &\textnormal{if } \Delta < 0,
        \end{cases}
\end{equation*}
and then $\tilde x_1 \coloneqq \min{\Bigl( 1, \frac{\sqrt{\tilde x_0}}{3} \Bigr)}$, we obtain that both of them only depend on $\Gamma$ and $\alpha$ and thus so is $\mu_0 = \frac{x_0}{3} x_1 - x_1^3$. Last, the conditions on $U $ becomes $U  \leq \delta \gamma_+ \, \mu_0$, which is in agreement with the invariances on $D$ (invariance by shifting of the $\gamma_i$s, invariance by multiplication of the $\gamma_i$s with respect to a change of time variable and a multiplication of the external field-control).

\bibliographystyle{plain} 
{\footnotesize
\bibliography{biblio}}

\begin{thebibliography}{10}

\bibitem{agarwal2011control}
Shruti Agarwal, Gilles Carbou, St{\'e}phane Labb{\'e}, and Christophe Prieur.
\newblock Control of a network of magnetic ellipsoidal samples.
\newblock {\em Mathematical Control and Related Fields}, 1(2):129--147, 2011.

\bibitem{MR2833256}
Shruti Agarwal, Gilles Carbou, St\'{e}phane Labb\'{e}, and Christophe Prieur.
\newblock Control of a network of magnetic ellipsoidal samples.
\newblock {\em Math. Control Relat. Fields}, 1(2):129--147, 2011.

\bibitem{alouges2009magnetization}
Fran{\c{c}}ois Alouges and Karine Beauchard.
\newblock Magnetization switching on small ferromagnetic ellipsoidal samples.
\newblock {\em ESAIM: Control, Optimisation and Calculus of Variations},
  15(3):676--711, 2009.

\bibitem{5399599}
Francois Alouges, Karine Beauchard, and Mario Sigalotti.
\newblock Magnetization switching in small ferromagnetic ellipsoidal samples.
\newblock In {\em Proceedings of the 48h IEEE Conference on Decision and
  Control (CDC) held jointly with 2009 28th Chinese Control Conference}, pages
  2106--2111, 2009.

\bibitem{MR4419351}
Xin An, Ananta~K. Majee, Andreas Prohl, and Thanh Tran.
\newblock Optimal control for a coupled spin-polarized current and
  magnetization system.
\newblock {\em Adv. Comput. Math.}, 48(3):Paper No. 28, 40, 2022.

\bibitem{brown1963micromagnetics}
William~Fuller Brown.
\newblock {\em Micromagnetics}.
\newblock Interscience, 1963.

\bibitem{MR2375581}
Gilles Carbou, St\'{e}phane Labb\'{e}, and Emmanuel Tr\'{e}lat.
\newblock Control of travelling walls in a ferromagnetic nanowire.
\newblock {\em Discrete Contin. Dyn. Syst. Ser. S}, 1(1):51--59, 2008.

\bibitem{di2016newtonian}
Giovanni Di~Fratta.
\newblock The newtonian potential and the demagnetizing factors of the general
  ellipsoid.
\newblock {\em Proceedings of the Royal Society A: Mathematical, Physical and
  Engineering Sciences}, 472(2190):20160197, 2016.

\bibitem{dubey2019controllability}
Shruti Dubey and Sharad Dwivedi.
\newblock On controllability of a two-dimensional network of ferromagnetic
  ellipsoidal samples.
\newblock {\em Differential Equations and Dynamical Systems}, 27(1):277--297,
  2019.

\bibitem{MR3407264}
Thomas Dunst, Markus Klein, Andreas Prohl, and Ailyn Sch\"{a}fer.
\newblock Optimal control in evolutionary micromagnetism.
\newblock {\em IMA J. Numer. Anal.}, 35(3):1342--1380, 2015.

\bibitem{MR3961301}
Thomas Dunst, Ananta~K. Majee, Andreas Prohl, and Guy Vallet.
\newblock On stochastic optimal control in ferromagnetism.
\newblock {\em Arch. Ration. Mech. Anal.}, 233(3):1383--1440, 2019.

\bibitem{hubert2008magnetic}
Alex Hubert and Rudolf Sch{\"a}fer.
\newblock {\em Magnetic domains: the analysis of magnetic microstructures}.
\newblock Springer Science \& Business Media, 2008.

\bibitem{landau2013electrodynamics}
Lev~Davidovich Landau, JS~Bell, MJ~Kearsley, LP~Pitaevskii, EM~Lifshitz, and
  JB~Sykes.
\newblock {\em Electrodynamics of continuous media}, volume~8.
\newblock Elsevier, 2013.

\bibitem{osborn1945demagnetizing}
John~A Osborn.
\newblock Demagnetizing factors of the general ellipsoid.
\newblock {\em Physical review}, 67(11-12):351, 1945.

\bibitem{parkin2008magnetic}
Stuart~SP Parkin, Masamitsu Hayashi, and Luc Thomas.
\newblock Magnetic domain-wall racetrack memory.
\newblock {\em Science}, 320(5873):190--194, 2008.

\bibitem{perko2013differential}
Lawrence Perko.
\newblock {\em Differential equations and dynamical systems}, volume~7.
\newblock Springer Science \& Business Media, 2013.

\bibitem{MR0186436}
L.~S. Pontryagin, V.~G. Boltyanskii, R.~V. Gamkrelidze, and E.~F. Mishchenko.
\newblock {\em The mathematical theory of optimal processes}.
\newblock A Pergamon Press Book. The Macmillan Company, New York, 1964.
\newblock Translated by D. E. Brown.

\bibitem{MR3348399}
Yannick Privat and Emmanuel Tr\'{e}lat.
\newblock Control and stabilization of steady-states in a finite-length
  ferromagnetic nanowire.
\newblock {\em ESAIM Control Optim. Calc. Var.}, 21(2):301--323, 2015.

\bibitem{takahashi2017ellipsoids}
Diego Takahashi and Vanderlei~C Oliveira~Jr.
\newblock {Ellipsoids (v1. 0): 3-D magnetic modelling of ellipsoidal bodies}.
\newblock {\em Geoscientific Model Development}, 10(9):3591--3608, 2017.

\bibitem{MR1773088}
Augusto Visintin.
\newblock Mathematical models of hysteresis. {A} survey.
\newblock In {\em Nonlinear partial differential equations and their
  applications. {C}oll\`ege de {F}rance {S}eminar, {V}ol. {XIII} ({P}aris,
  1994/1996)}, volume 391 of {\em Pitman Res. Notes Math. Ser.}, pages
  327--340. Longman, Harlow, 1998.

\bibitem{MR3011326}
Baisheng Yan.
\newblock On energy-minimization in ferromagnetism controlled by applied
  fields.
\newblock {\em Ann. Mat. Pura Appl. (4)}, 192(1):115--125, 2013.

\end{thebibliography}

\end{document}